\documentclass[11pt]{article}

\usepackage[utf8]{inputenc}
\usepackage[T1]{fontenc}
\usepackage[main=english]{babel}
\usepackage{lmodern}
\usepackage{tcolorbox}

\usepackage{graphicx}
\usepackage{tikz}
\usetikzlibrary{arrows}
\usetikzlibrary{decorations}
\usepackage{tikz-3dplot}
\usepackage{pgfplots}
\usepackage{hyperref}
\usetikzlibrary{patterns}
\usetikzlibrary{arrows}
\usepackage{caption}

\usepackage{a4wide}
\usepackage{fancyhdr}
\usepackage{amsmath,amssymb,amsthm}
\usepackage{mathtools}
\usepackage[mathscr]{euscript}

\newtheorem{definition}{Definition}
\newtheorem{remark}{Remark}
\newtheorem{theorem}{Theorem}
\newtheorem{lemma}{Lemma}

\newtheorem{assumption}{Assumption}
\graphicspath{{images/}}

\newcommand\p[1]{\boldsymbol{#1}}
\newcommand\norm[1]{\left\lVert#1\right\rVert}

\title{A space-time isogeometric method for the partial differential-algebraic system
	of Biot's poroelasticity model}

\author{Jeremias Arf \footnotemark[2]
	\and Bernd Simeon \footnotemark[2]}

\providecommand{\keywords}[1]
{
	\small	
	\textbf{\textit{Keywords---}} #1
}

\begin{document}
	
	\maketitle
	\keywords{Biot's poroelasticity model, isogeometric analysis, space-time discretization, high-order convergence}
	\renewcommand{\thefootnote}{\fnsymbol{footnote}}
	
	\footnotetext[2]{ TU Kaiserslautern, Dept. of Mathematics, Gottlieb-Daimler-Str. 48, 67663 Kaiserslautern,\\ Germany ( E-mail: {\tt \{arf,simeon\}@mathematik.uni-kl.de}).}
	
\begin{abstract}
	Biot's equations of poroelasticity contain a parabolic system for the evolution of the pressure, which is coupled with a quasi-stationary equation for the stress tensor. Thus, it is natural to extend the existing work  on isogeometric space-time methods to this more advanced framework of a partial differential-algebraic equation (PDAE). 
	A space-time approach based on finite elements has already been  introduced. But we present a new weak formulation in space and time that is appropriate for an isogeometric discretization and    
	analyze the convergence properties.
	Our approach is based on a single variational problem and hence differs from the iterative space-time schemes considered so far. Further, it enables high-order convergence. Numerical  
	experiments that have been carried out 
	confirm the theoretical findings.
\end{abstract}


	\section{Introduction}
	\label{sec:introduction}
	Poroelasticity describes the coupling of   mechanical deformation with the flow of a fluid in a porous medium and has numerous applications in engineering. In the quasi-static case, which is most widely adopted, the model takes the form of a partial differential-algebraic equation (PDAE) that calls for appropriate discretizations in space and time. The conventional approach relies so far on  the method of lines, starting with a semi-discretization by finite elements in space followed by a time integration 
	of the resulting differential-algebraic system. We introduce here a novel scheme that treats space and time simultaneously and that applies isogeometric analysis (IGA) as discretization.
	In this way, we extend the framework of space-time methods to PDAE problems and show how 
	the powerful algorithmic machinery of IGA with its spline-based function spaces can be applied to such 
	coupled models.
	
	For more background on poroelasticity, we refer to the example of  reservoir  engineering  where
	oil and gas reservoirs,  as well as more sustainable energy resources like geothermal reservoirs,  are subject of  research \cite{reservoi_Biot1,Augustin,Biot_reservoir_finite}. This includes the problem of induced seismicity caused by the injection or extraction of fluids in the subsurface of the earth, leading to anthropogenic earthquakes \cite{eartquake_damage1,eartquake_damage2}. Further references on poroelasticity in earthquake engineering are \cite{earthquake_biot2,earthquake_biot4,earthquake_biot5}.  
	The works of Karl von Terzaghi \cite{Terzaghi} and Maurice Anthony Biot \cite{Biot1935,BiotGeneral,Biot1955TheoryOE} laid the foundation for poroelasticity and were triggered by the observation of \emph{consolidation}, meaning the  volume decrease of a fluid-saturated soil caused by an applied loading and fluid discharge. 
	Today, biophysics and biomedicine also make use of poroelastic models, see, e.g., \cite{Biology_Biot1} for the mechanical modelling of living  tissues and \cite{Biology_Biot3} for a  model of the fluid-structure interaction in the human brain.  
		
	Our article concentrates on the quasi-static Biot system where the deformation is assumed to be much slower than the fluid flow rate. The corresponding coupled model 
	and its numerical solution has been studied before. Without aiming at completeness, we mention  
	the finite element methods (FEM) for the Biot equations introduced in \cite{Biot_finite_element_1}, where Lagrange, Taylor-Hood and MINI elements are applied. The authors in  \cite{Biot_wheeler_continuous,Biot_wheeler_discrete} favor a method involving Raviart-Thomas elements. Further finite element approaches can be found in \cite{Biot_finite_element_2,Biot_finite_element_3,Biot_finite_element_4} and in the dissertations \cite{phillips,Biot_finite_element_dissertation}. 
	Note that there are different versions of the Biot model in the sense that the underlying set of primary unknown variables might differ. For example in \cite{Biot_two_filed} the Biot two-field system, in \cite{Biot_finite_element_2} a Biot three-field and in \cite{Biot_finite_element_3} a four-field  formulation are used. The two main reasons for the varying systems are on the one hand the possibility to improve stability properties of the numerical method. On the other hand sometimes specific variables like the fluid flux are of exceptional interest and a formulation in which such variables are incorporated is preferred. But three- or four-field formulations suffer from the drawback of an increased number of degrees of freedom. 

	A different class of methods are the ones based on IGA. In IGA, geometric approximations are avoided, and at the same
	time the underlying spline spaces offer various properties, including increased global smoothness. Isogeometric methods for the Biot system are considered in \cite{Biotmixed} and \cite{BiotIGA}. 
	An iterative method for solving Biot's model numerically based on a finite element discretization is introduced in \cite{BauseKoecher}. In the latter reference the authors analyze a space-time  scheme, i.e. the time variable itself is discretized with continuous or discontinuous finite elements.
	
	In this paper we also follow the idea of space-time discretizations and use the approach of \cite{para} for a parabolic
	evolution equation as starting point. The extension to 
	the Biot two-field system is not straightforward since we are facing here several challenges. First, an appropriate weak formulation for both pressure and displacement variables has to be derived. Second, the treatment of the elastic momentum balance requires special care as it represents a constraint with respect to the time axis. And third, finally, the 
	pressure variable is known to be sensitive to oscillations that call for additional measures.
	A space-time method leads directly to a linear system that includes all time steps and hence is much larger than 
	a spatial discretization alone. However, by closer inspection of the linear system one notices a staircase structure that reflects the propagation of the solution over time and that results in a very sparse matrix. With appropriate fast iterative 
	and parallel solvers, the linear system can be tackled as a whole, but in our work, the focus is on the
	discretization itself. We also point out that in principal, the space-time scheme can be employed to introduce full adaptivity 
	in space and time, which comprises in particular local time step changes in combination with local spatial refinement. 
	
	The outline of the paper is as follows: 
	Section 2 explains briefly the Biot system while in Section 3 we introduce the space-time method and the underlying space-time discretization. Section 4 presents a convergence result,  
	and in the last part we discuss numerical examples. 
	
	The notation that we use is fairly standard. We write for scalar Sobolev spaces over an open domain $D$ just  $L^2(D)= H^0(D)$, $H^k(D)$ for some $k \in \mathbb{N}$. The standard scalar product of the Lebesgue space $L^2(D)$ is denoted by $\langle \cdot , \cdot \rangle_{L^2(D)}$.  In case of vector-valued Sobolev spaces we use a bold type notation, for example $\p{H}^k(D) \coloneqq (H^k(D), \dots , H^k(D))$ etc. Moreover,  $\norm{\cdot}_{L^2(D)}, \ \norm{\cdot}_{\p{L}^2(D)}, \ \norm{\cdot}_{H^k(D)}, \ \norm{\cdot}_{\p{H}^k(D)}$ stand for the norms induced by the inner products in the respective spaces.  Finally, $\nabla_x$ denotes the classical nabla operator in the spatial coordinates.


\section{The Biot system}
\label{sec:Biot_system}
In this section we outline the coupled model of Biot, following the references
\cite{Augustin,Biot_wheeler_continuous} and \cite{Showalter}.
The porous medium   consists of a solid skeleton  and permeable voids (pores) that are filled by some fluid. It is identified with a connected and bounded {Lipschitz domain} $\Omega \subset \mathbb{R}^d, \ d \in \{1,2,3\}$. 
Model variables  are the time-dependent mechanical displacement field $\p{u}(t) \colon \Omega \rightarrow \mathbb{R}^d$ as well as the fluid pore pressure $p(t)\colon \Omega \rightarrow \mathbb{R}$ and the fluid flux $\ \p{q}(t)\colon \Omega \rightarrow \mathbb{R}^d$, with $t$ as the time. Flux and pressure are connected via \emph{Darcy's law}
\begin{align*}
\p{q} = -\frac{\p{K}}{\eta_f} \big(\nabla_x p-\rho_f\p{g}\big), \,
\end{align*}
which gives a linear relation between  the pressure gradient and the flux. In the last equation $\p{K}$ denotes the permeability tensor, represented by a spd (symmetric positive-definite) matrix, $\eta_f$ the fluid viscosity, $\rho_f$ the fluid density and $\p{g}$ some body force. For simplicity, we drop the body force from now on.  

The first governing equation of the Biot system can be derived by means of the balance of momentum and linear elasticity, namely
\begin{align}
\label{eq:balance _of_momentum}
-\nabla_x \cdot  \p{\tilde{\sigma}}(\p{u},p)=\p{f} , \,
\end{align} 
where $\p{f}$ is a given force distribution.
In this equation, $\p{\tilde{\sigma}}= \p{\sigma}-b \,  p \,  \p{I}$ denotes the total stress tensor and $b$ is the so-called Biot-Willis constant. Assuming a linear-elastic behavior, the solid phase satisfies Hooke's law $
{{\sigma}_{ij}}=C_{ijkl} {{\varepsilon}_{kl}}$ (Einstein not.)
with $\p{C}=(C_{ijkl})$ being the elasticity tensor, 
$\p{\varepsilon}(\p{u}) = (\nabla_x \p{u} +\nabla_x \p{u}^t)/2$ the strain tensor and  
$\p{\sigma}$ the stress tensor. 
Furthermore, we assume a quasi-static behavior where second time derivatives are neglected. For the special case of an isotropic and homogeneous solid we can simplify the elasticity tensor to $C_{ijkl}= \lambda \, (\delta_{ij} \delta_{kl}) + \mu \, (\delta_{ik}\delta_{jl}+\delta_{il}\delta_{jk}),$ where $\lambda$ and $\mu$ are the  Lam\'{e} constants and $\delta$ the Kronecker delta.   

The second governing equation is obtained from a conservation law
for the fluid phase, which reads
\begin{align}
\label{eq:conservation_of_mass}
\partial_t \big( c_0 p + b \nabla_x \cdot \p{u} \big) + \nabla_x \cdot \p{q} =  g.
\end{align}
Here, we can interpret $c_0 p + b \nabla_x \cdot \p{u}$ as the fluid content and $g$ as a fluid source term. The constant $c_0$ is the \emph{constrained specific storage coefficient} and  in applications it is often close to zero.
Both, \eqref{eq:balance _of_momentum} and  \eqref{eq:conservation_of_mass}  together  with Darcy's law lead to the 
{\emph{Biot two-field system}:}
\begin{align}
\label{Eq1}
-\nabla_x \cdot  \p{\tilde{\sigma}}(\p{u},p)=\p{f},  
\\ \label{Eq2}
\partial_t \big( c_0 p + b \nabla_x \cdot \p{u} \big) - \nabla_x \cdot \frac{\p{K}}{\eta_f} \nabla_x p =  g. 
\end{align}
As initial conditions we require $p(0) = 0 $ and $\p{u}(0)=\p{0}$, i.e. pressure and displacement are zero at the start time $t=0$. For the boundary conditions we introduce two partitions $\overline{\Gamma_u} \cup\overline{ \Gamma_t} = \partial \Omega $ and $\overline{\Gamma_p} \cup\overline{ \Gamma_f} = \partial \Omega $ of the boundary $\partial \Omega$ of the spatial domain with $\Gamma_u \cap \Gamma_t = \emptyset$ and  $\Gamma_p \cap \Gamma_f = \emptyset$. Then we choose
\begin{alignat}{3}
\label{init_bound_con_1}
\p{\tilde{\sigma}} \cdot \p{n}_x &= \p{t}_n   \hspace{0.5cm} &&\textup{on}  \hspace{0.5cm} \Sigma_t &&\coloneqq \Gamma_t  \times (0,T),  \ \ \ \ \ (\textup{tension BC.})\\
\vspace{0.06cm} \label{init_bound_con_2}
  \p{\mathcal{K}} \nabla_xp \cdot \p{n}_x &= v_f  \hspace{0.5cm} &&\textup{on} \hspace{0.5cm}  \Sigma_f &&\coloneqq \Gamma_f   \times (0,T), \ \ \ \ \ (\textup{flux BC.})\\
\vspace{0.06cm} \label{init_bound_con_3}
 \p{u} &= 0 \hspace{0.5cm} &&\textup{on} \hspace{0.5cm} \Sigma_u &&\coloneqq \Gamma_u \times (0,T), \\
\vspace{0.06cm} \label{init_bound_con_4}
 p &= 0  \hspace{0.5cm} &&\textup{on} \hspace{0.5cm} \Sigma_p &&\coloneqq  \Gamma_p  \times (0,T).
\end{alignat}
Here $I \coloneqq (0,T)$ is the time interval of interest and $\p{\mathcal{K}} \coloneqq \eta_f^{-1}  \p{K}$. 
Moreover, $\p{n}_x$ denotes the outer unit normal vector.

A study on the existence of (weak) solutions to the Biot system can be found in \cite{Showalter}. 
To show the existence of discrete solutions within the scope of our proposed method we have to postulate the next assumption.
\begin{assumption}
{\rm	Let the boundary parts $\Gamma_u$ and $\Gamma_p$ have positive measure, meaning $0 < \int_{\Gamma_u} 1 \, ds, \ \int_{\Gamma_p} 1 \, ds$. Besides, let the  tensor $\p{\mathcal{K}}(\p{x})$ be uniformly elliptic and bounded  in $\Omega$. And w.l.o.g. we assume $c_0 \leq 1$. }
\end{assumption}

The most important model and material parameters are summarized in the tables Tab. \ref{tab1} and Tab. \ref{tab2} below.
\begin{table}[t]
	\renewcommand{\arraystretch}{1.2}
	\renewcommand{\footnoterule}{\rule{0pt}{0pt}}
	\begin{minipage}[t]{8.7cm}
		\caption{\label{tab1}{Material parameters}}
		\vspace{1.2ex} \footnotesize
		\begin{tabular}{c c c c} 
			\hline\hline 
			Parameter & phy. unit &   Explanation   \\  \hline \\ [\dimexpr-\normalbaselineskip+3pt]   %
			$\p{K}$ & $[m^2]$ &  permeability tensor\vspace{0.1cm} \\
			$0<\eta_f$ & $[N \, s/m^2]$ &   viscosity of the fluid  \vspace{0.1cm} \\
			$0<\rho_f$ & $[kg/m^3]$ & density of the fluid \vspace{0.1cm} \\ 
			$0<b$ & $[1]$ &  Biot-Willis coefficient \vspace{0.1cm} \\
			$0\leq c_0 $ & $[m^2/N]$ &  constrained specific storage   \\ [1ex] \hline 
		\end{tabular}
	\end{minipage}
	\hspace{0cm}
	\begin{minipage}[t]{6.5cm}
		\caption{\label{tab2}{Model quantities}}
		\vspace{1.2ex} \footnotesize
		\begin{tabular}{c c c c} 
			\hline\hline 
			Variable & phy. unit &   Explanation  \\  \hline \\ [\dimexpr-\normalbaselineskip+3pt] 
			$\p{\tilde{\sigma}}$ & $[N/m^2]$ &  total stress tensor \vspace{0.1cm} \\
			$\p{C}$ & $[N/m^2]$ &  elasticity tensor \vspace{0.1cm} \\ 
			$\p{\varepsilon}$ & $[1]$ &  lin. strain tensor  \vspace{0.1cm} \\
			$\p{q}$ & $[m/s]$ &  volumetric flux  \vspace{0.1cm} \\
			$\p{u}$ & $[m]$ &  displacement \vspace{0.1cm} \\
			$p$ & $[N/m^2]$ &  fluid pressure \\ [1ex] 
			\hline 
		\end{tabular}
	\end{minipage}
\end{table}
\normalsize


\section{Space-time discretization and discrete variational formulation}
\label{sec:space_time_discretization}
Regarding the discretization of the Biot system, we look first at the basics of isogeometric analysis and proceed with the 
derivation of a discrete space-time variational method for it.

\subsection{Isogeometric analysis}
\label{subsec:IGA}
Introduced by Hughes et al. \cite{IGA2}, the concept of IGA developed in the last 15 years to a powerful tool in numerical analysis. The basic idea is the simultaneous use of spline functions for the geometric modelling and the definition of discrete spaces. IGA is able to represent various complex and curved-boundary domains exactly and furthermore one has the possibility to easily increase or lower the smoothness of functions in the discrete spaces. Following \cite{IGA1,IGA3,IGA4} for a brief
exposition, we call an  increasing sequence of real numbers $\Xi \coloneqq \{ \xi_1 \leq  \xi_2  \leq \dots \leq \xi_{n+r+1}  \}$ for some $r \in \mathbb{N}$   \emph{knot vector}, where we assume  $0=\xi_1=\xi_2=\dots=\xi_{r+1}, \ \xi_{n+1}=\xi_{n+2}=\dots=\xi_{n+r+1}=1$, and call such knot vectors $r$-open. 
Further, the multiplicity of the $j$-th knot is denoted by $m(\xi_j)$.
Then  the univariate B-spline functions $\hat{B}_{j,r}(\cdot)$ of degree $r$ corresponding to a given knot vector $\Xi$ is defined recursively by the \emph{Cox-DeBoor formula}:

\begin{align*}
\hat{B}_{j,0}(\zeta) \coloneqq \begin{cases}
1, \ \ \textup{if}  \ \zeta \in [\xi_{j},\xi_{j+1}) \\
0, \ \ \textup{else},
\end{cases}
\end{align*}
\textup{and if} $r \in \mathbb{N}_{\geq 1} \ \textup{we set}$ 
\begin{align*}
\hat{B}_{j,r}(\zeta)\coloneqq \frac{\zeta-\xi_{j}}{\xi_{j+r}-\xi_j} \hat{B}_{j,r-1}(\zeta)  +\frac{\xi_{j+r+1}-\zeta}{\xi_{j+r+1}-\xi_{j+1}} \hat{B}_{j+1,r-1}(\zeta),
\end{align*}
where one puts $0/0=0$ to obtain  well-definedness. The multivariate extension of the last spline definition is achieved by a tensor product construction. In other words, we set for a given tensor knot vector   $\p{\Xi} \coloneqq \Xi_1 \times   \dots \times \Xi_d $, where the $\Xi_{l}=\{ \xi_1^{l}, \dots , \xi_{n_l+r_l+1}^{l} \}, \ l=1, \dots , d$ are $r_l$-open,   and a given \emph{degree vector}   $\p{r} \coloneqq (r_1, \dots , r_d)$ for the multivariate case
\begin{align*}
\hat{B}_{\p{i},\p{r}}(\p{\zeta}) \coloneqq \prod_{l=1}^{d} \hat{B}_{i_l,r_l}(\zeta_l), \ \ \ \ \p{\zeta} \coloneqq (\zeta_1,\dots,\zeta_d), \ \  \forall \, \p{i} \in \mathit{\mathbf{I}}, 
\end{align*}
with  $d$ as  the underlying dimension of the parametric domain $\hat{\Omega}= (0,1)^d$ and $\mathit{\mathbf{I}}$ the multi-index set $\mathit{\mathbf{I}} \coloneqq \{ (i_1,\dots,i_d) \  | \  1\leq i_l \leq n_l, \ l=1,\dots,d  \}$.
To enlarge the possibilities of the representation of geometric objects, one can generalize the definition of B-splines 
to rational B-splines. Namely, choosing strictly positive weights $0<w_{\p{i}}, \p{i} \in \mathit{\mathbf{I}}$ and exploiting the notation from above we introduce the weight function $$W(\p{\zeta}) \coloneqq \sum_{\p{i}\in \mathit{\mathbf{I}}} w_{\p{i}} \, \hat{B}_{\p{i},\p{r}}(\p{\zeta}).$$
We define the non-uniform rational B-spline (NURBS) basis functions $\hat{N}_{\p{i},\p{r}}(\p{\zeta})$ w.r.t. to the weight function $W$ as follows:
\begin{align*}
\hat{N}_{\p{i},\p{r}}(\p{\zeta}) \coloneqq w_{\p{i}} \, \hat{B}_{\p{i},\p{r}}(\p{\zeta}) \big(W(\p{\zeta})\big)^{-1}  , \ \ \ \ \forall \, \p{i} \in \mathit{\mathbf{I}}.
\end{align*}
B-splines (and the same for NURBS) fulfil several properties and for our purposes the most important ones are:
\begin{itemize}
	\item If for all internal knots the multiplicity satisfies $1 \leq m(\xi_j^l) \leq m \leq r \leq r_l, \ \forall l $, then the B-spline basis functions $\hat{B}_{\p{i},\p{r}}$ are globally $C^{r-m}$-continuous.
	\item The B-splines $ \{\hat{B}_{\p{i},\p{r}} \ | \ \ \p{i} \in \mathit{\mathbf{I}} \}$ are linearly independent.
\end{itemize}

Back to the Biot problem, the aim is the definition of a space-time discretized variational formulation. Consequently we consider  as in \cite{para} the space-time cylinder  $\mathcal{Q}=\Omega \times (0,T)$. 
This so-called physical domain is assumed to be parametrized by means of NURBS or B-splines, respectively.
More precisely, we have a parametrization of the form 
$$\tilde{\p{\Phi}} \colon (0,1)^{d+1} \eqqcolon \hat{\mathcal{Q}} \rightarrow \mathcal{Q} \ , \ \tilde{\p{\zeta}} \mapsto   \sum_{\tilde{\p{i}} \in {\tilde{\mathit{\mathbf{I}}}} } \tilde{C}_{\tilde{\p{i}}} \, \hat{N}_{\tilde{\p{i}},\tilde{\p{r}} }(\tilde{\p{\zeta}}) \ ,
$$ 
where  the $\tilde{C}_{\tilde{\p{i}}}=(C_{\p{i}},t_{i_{d+1}})  \in \mathbb{R}^{d+1}$ are the \emph{control points} and  $\tilde{\mathit{\mathbf{I}}}=\{(i_1, \dots,i_d,i_{d+1} ) \in \mathit{\mathbf{I}} \times \textup{I}_t\}$, \ $\tilde{\p{\zeta}}=(\p{\zeta},\zeta_{d+1}),$ and $ \ \tilde{\p{r}}=(\p{r},r_{d+1})$ for suitable index sets $\p{\textup{I}}$ and $\textup{I}_t$.
Due to the  product structure of the space-time cylinder we can assume that the parametrization can be written as 

\begin{align}
\label{eq:tensor_structure}
&\tilde{\p{\Phi}}  \colon \hat{\mathcal{Q}} \rightarrow \Omega  \times (0,T) \ , \ \tilde{ \p{\zeta}} \mapsto  \big(\p{\Phi}(\p{\zeta}) \ , \  \Phi(\zeta_{d+1}) \big),  \\[1mm]  \textup{with} \hspace{0.5cm} &\p{\Phi}(\p{\zeta}) = \sum_{\p{i} \in \mathit{\mathbf{I}}} C_{\p{i}} \, \hat{N}_{\p{i},\p{r} }(\p{\zeta}) \hspace{0.5cm}  \textup{and} \hspace{0.5cm} \Phi(\zeta_{d+1}) = \sum_{i \in \textup{I}_t} t_i \, \hat{B}_{i,r_{d+1} }(\zeta_{d+1})= T \, \zeta_{d+1}. \nonumber
\end{align}

Given such a parametrization the knots stored in the knot vector $\tilde{\p{\Xi}} \coloneqq \p{\Xi}  \times \Xi_{d+1}$, corresponding to  the underlying NURBS and splines, determine a mesh in the parametric domain $\hat{\mathcal{Q}}$, namely  $\hat{M} \coloneqq \{ K_{\p{j}}\coloneqq (\psi_{j_1}^1,\psi_{j_1+1}^1 ) \times \dots \times (\psi_{j_{d+1}}^{d+1},\psi_{j_{d+1}+1}^{d+1} ) \ | \  \p{j}=(j_1,\dots,j_{d+1}), \ \textup{with} \ 1 \leq j_i <n_i\},$ and
with $\tilde{\p{\Psi}}= \{\psi_1^1, \dots, \psi _{n_1}^1\}  \times \dots \times \{\psi_1^{d+1}, \dots, \psi _{n_{d+1}}^{d+1}\}$  \  \textup{as  the knot vector} \ $\tilde{\p{\Xi}}$ \ 
\textup{without knot repetitions}. 

The image of this mesh under the mapping $\tilde{\p{\Phi}}$, i.e. $\mathcal{M} \coloneqq \{\tilde{\p{\Phi}}(\hat{K}) \ | \ \hat{K} \in \hat{M} \}$, gives us a mesh structure in the physical domain. By inserting knots without changing the parametrization  we can refine the mesh, which is the concept of $h$-refinement \cite{IGA2,IGA1}. Furthermore, we can introduce the mesh size  $h \coloneqq \max\{h_{\mathcal{K}} \ | \ \mathcal{K} \in \mathcal{M} \}$, where  $h_{\mathcal{K}} = \textup{diam}(\mathcal{K})$ is the diameter of the mesh element $\mathcal{K}$.
For the rest of this article, we assume the mesh to be regular as defined next:

\begin{assumption}{(Regular mesh)}\\
	The parametrization mapping is smooth on the closure of each mesh element $\overline{\hat{K}}, \ \hat{K} \in \hat{M}$ and has a smooth inverse, meaning $\tilde{\p{\Phi}}_| \in C^{\infty}(\overline{\hat{K}})$, \ $\tilde{\p{\Phi}}^{-1}_| \in C^{\infty}(\tilde{\p{\Phi}}(\overline{\hat{K}}))$.  \\
	Further we can find a constant $0< c_M < \infty$, independent from mesh refinement, such hat for the element sizes it holds  $h_{\mathcal{K}} \leq h \leq   c_M \, h_{\mathcal{K}} $ for all mesh elements $\mathcal{K} \in \mathcal{M}$. \\
	And for the coarsest mesh the boundary segments $\Sigma_u$ and $\Sigma_p$ are the unions of full  boundary mesh faces. 
\end{assumption}

Clearly the global mesh is composed of a spatial mesh and a mesh in the time interval $(0,T)$ as consequence of the product structure. Therefore one can introduce a spatial mesh size $h_S$ and a mesh size in the time domain $h_T$ in an analogous manner. Although the  mesh sizes $h_S, \ h_T$ are more convenient for our considerations we also keep the global mesh size $h$  to shorten the notation.

Lastly, we define the discrete spaces, following the isogeometric paradigm, which are used below for the discretized variational formulation via
$$\mathcal{V}_{h, \tilde{\p{r}}} \coloneqq \textup{span}\{ v_{h}=\hat{N}_{\tilde{\p{i}},\tilde{\p{r}} } \circ \tilde{\p{\Phi}}^{-1} \ | \  \tilde{\p{i}} \in \tilde{\mathit{\mathbf{I}}}  \},$$
spanned by the  push-forwards of the NURBS basis functions.

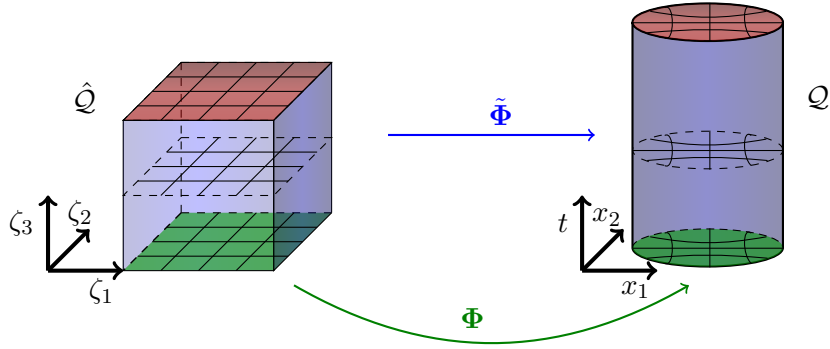
\begin{figure}
	\begin{center} 
		\begin{tikzpicture}[scale=0.5]
		\fill[left color=black!50!blue,right color=black!50!blue,middle color=blue,shading=axis,opacity=0.25] (2,0) -- (2,6) arc (360:180:2cm and 0.5cm) -- (-2,0) arc (180:360:2cm and 0.5cm);
		\draw[thick] (-2,6) -- (-2,0) arc (180:360:2cm and 0.5cm) -- (2,6) ++ (-2,0) circle (2cm and 0.5cm);
		\fill[top color=black!50!red,bottom color=red,middle color=red!90!white,shading=axis,opacity=0.35 ] (0,6) circle (2cm and 0.5cm);
		\draw[] (0,6) circle (2cm and 0.5cm);
		\fill[top color=green,bottom color=green,middle color=green,shading=axis,opacity=0.35] (0,0) circle (2cm and 0.5cm);
		\draw[densely dashed] (-2,0) arc (180:0:2cm and 0.5cm);

		\draw[out=30,in=120,->, thick, blue] (-8.5,3,0) -- (-3,3,0);
		\node[above,blue] at (-5.5,3,0) {$\tilde{\p{\Phi}}$};
		x={(2-14,0)},
		y={(0-14,2)},
		z={({2*cos(45)},{2*sin(45)})},
		]
		\def\a{4}
		
		\coordinate (A) at (0-14,0+0.93,0); 
		\coordinate (B) at (\a-14,0+0.93,0) ;
		\coordinate (C) at (\a-14,\a+0.93,0); 
		\coordinate (D) at (0-14,\a+0.93,0); 
		\coordinate (E) at (0-14,0+0.93,\a); 
		\coordinate (F) at (\a-14,0+0.93,\a); 
		\coordinate (G) at (\a-14,\a+0.93,\a); 
		\coordinate (H) at (0-14,\a+0.93,\a);
		\draw[] (E) -- (H)--(G)--(F);
		\draw[left color=black!50!blue,right color=black!50!blue,middle color=blue,shading=axis,opacity=0.25] (E) -- (H)--(G)--(F);
		\fill[top color=green,bottom color=green,middle color=green,shading=axis,opacity=0.35] (0-14,0+0.93,0) -- (4-14,0+0.93,0) -- (4-14,0+0.93,4) -- (0-14,0+0.93,4) -- (0-14,0+0.93,0);
		\node[above] at (-0.30-16,0+0.7,1) {${\zeta}_2$};
		\node[above] at (1.50-16,-1+0.85,4.2) {${\zeta}_1$};
		\node[left] at (-1.8-15.8,0+0.7,0.1) {${\zeta}_3$};
		
		\draw[dashed] (A) 
		-- (B) ;
		\draw[dashed] (A)--(D);
		\draw (D)--(C);
		\draw[dashed] (A)--(E);
		\draw (B)--(C);
		\draw (C)--(D);
		\draw[dashed] (A)--(E);
		\fill[left color=black!50!blue,right color=black!50!blue,middle color=blue,shading=axis,opacity=0.25] (B) -- (F) -- (G) -- (C);
		\draw(B) -- (F) -- (G) -- (C);
		\fill[top color=black!50!red,bottom color=red,middle color=red!90!white,shading=axis,opacity=0.35] (C)--(G) -- (H) -- (D);
		\draw[] (C)--(G) -- (H) -- (D);
		\draw[] (E) -- (F);
		\draw[line width=1.55pt, ->, black] (0-16,0+0.93,4) -- (0-16,2+0.93,4);
		\draw[line width=1.55pt, ->, black] (0-16,0+0.93,4) -- (2-16,0+0.93,4);
		\draw[line width=1.55pt, ->, black] (0-16,0+0.93,4) -- (1.1-16,1.1+0.93,4);
		
		\draw[line width=1.55pt, ->, black] (-1.8,0+0.93,4) -- (-1.8,2+0.93,4);
		\draw[line width=1.55pt, ->, black] (0-1.8,0+0.93,4) -- (2-1.8,0+0.93,4);
		\draw[line width=1.55pt, ->, black] (0-1.8,0+0.93,4) -- (1.1-1.8,1.1+0.93,4);
		
		\node[above] at (-0.30-2,0+0.7,1) {$x_2$};
		\node[above] at (1.50-1.8,-1+1,4.2) {$x_1$};
		\node[left] at (-1.8-1.6,0+0.7,0.1) {$t$};
		
		\draw[thick,->,green!50!black,out=-30,in=-150] (-11,-1) to (-0.5,-1);
		\node[green!50!black,above] at (-6.25,-2.4) {$\p{\Phi}$};
		\node[left] at (-16,4) {$ \hat{\mathcal{Q}}$};
		\node[left] at (3.5,4) {$\mathcal{Q}$};
		\draw[]  (-14+1,0.93,0) to (-14+1,0.93,4);
		\draw[]  (-14+3,0.93,0) to (-14+3,0.93,4);
		\draw[]  (-14+2,0.93,0) to (-14+2,0.93,4);
		\draw[]  (-14,0.93,1) to (-14+4,0.93,1);
		\draw[]  (-14,0.93,2) to (-14+4,0.93,2);
		\draw[]  (-14,0.93,3) to (-14+4,0.93,3);
		\draw[out = -10, in = -170]  (-1+0.01,0.93-0.1,1.4) to (2.5-1+0.6,0.93-0.1,1.4);
		\draw[]  (-1.05,0.93,2.45) to (2.95,0.93,2.45);
		\draw[out = 5, in = 175]  (-1+0.7,0.93+0.1,3.4) to (3-1+0.9,0.93+0.1,3.4);
		\draw[]  (1,0.93+0.53,2.45) to (1,0.93-0.51,2.45);
		\draw[out = 30, in = -30]  (-0.2,0.93+0.4,2.45) to (-0.2,0.93-0.4,2.45);
		\draw[out = 140, in = 140]  (2.1,0.93+0.4,2.45) to (2.1,0.93-0.4,2.45);

		\draw[]  (-14+1,2.93,0) to (-14+1,2.93,4);
		\draw[]  (-14+3,2.93,0) to (-14+3,2.93,4);
		\draw[]  (-14+2,2.93,0) to (-14+2,2.93,4);
		\draw[]  (-14,2.93,1) to (-14+4,2.93,1);
		\draw[]  (-14,2.93,2) to (-14+4,2.93,2);
		\draw[]  (-14,2.93,3) to (-14+4,2.93,3);
		
		\draw[]  (-14+1,4.93,0) to (-14+1,4.93,4);
		\draw[]  (-14+3,4.93,0) to (-14+3,4.93,4);
		\draw[]  (-14+2,4.93,0) to (-14+2,4.93,4);
		\draw[]  (-14,4.93,1) to (-14+4,4.93,1);
		\draw[]  (-14,4.93,2) to (-14+4,4.93,2);
		\draw[]  (-14,4.93,3) to (-14+4,4.93,3);
		
		\draw[dashed] (0-14,2+0.93,0) --  (4-14,2+0.93,0) -- (4-14,2+0.93,4) -- (0-14,2+0.93,4) -- (0-14,2+0.93,0);
		
		\draw[dashed] (0-14,4+0.93,0) --  (4-14,4+0.93,0) -- (4-14,4+0.93,4) -- (0-14,4+0.93,4) -- (0-14,4+0.93,0);

		\draw[out = -10, in = -170]  (-1+0.01,0.93-0.1+2.6,1.4) to (2.5-1+0.6,0.93-0.1+2.6,1.4);
		\draw[]  (-1.05,0.93+2.6,2.45) to (2.95,0.93+2.6,2.45);
		\draw[out = 5, in = 175]  (-1+0.7,0.93+0.1+2.6,3.4) to (3-1+0.9,0.93+0.1+2.6,3.4);
		\draw[]  (1,0.93+0.53+2.6,2.45) to (1,0.93-0.51+2.6,2.45);
		\draw[out = 30, in = -30]  (-0.2,0.93+0.4+2.6,2.45) to (-0.2,0.93-0.4+2.6,2.45);
		\draw[out = 140, in = 140]  (2.1,0.93+0.4+2.6,2.45) to (2.1,0.93-0.4+2.6,2.45);
		\draw[dashed] (0,2.6) circle (2cm and 0.5cm);

		\draw[out = -10, in = -170]  (-1+0.01,0.93-0.1+6,1.4) to (2.5-1+0.6,0.93-0.1+6,1.4);
		\draw[]  (-1.05,0.93+6,2.45) to (2.95,0.93+6,2.45);
		\draw[out = 5, in = 175]  (-1+0.7,0.93+0.1+6,3.4) to (3-1+0.9,0.93+0.1+6,3.4);
		\draw[]  (1,0.93+0.53+6,2.45) to (1,0.93-0.51+6,2.45);
		\draw[out = 30, in = -30]  (-0.2,0.93+0.4+6,2.45) to (-0.2,0.93-0.4+6,2.45);
		\draw[out = 140, in = 140]  (2.1,0.93+0.4+6,2.45) to (2.1,0.93-0.4+6,2.45);
		\draw[dashed] (0,6) circle (2cm and 0.5cm);

		\end{tikzpicture}
	\end{center}\label{fig:space_time_domain}
	\caption{ \small We exploit the tensor product structure of $\mathcal{Q}$ for parametrizing the space-time cylinder.}
\end{figure}
For simplicity we assume the same polynomial degree  in each spatial coordinate direction and write $r_T$ for the polynomial degree w.r.t. the time parameter. Based on this we can define the test spaces for the pressure $p$ and the displacement $\p{u}$. Let $r_S \in \mathbb{N}_{\geq 1}$. We set 
\begin{align*}
\mathcal{V}_{h,r_S,r_T} \coloneqq  \mathcal{V}_{h, \tilde{\p{r}}}  \hspace{1cm} \textup{with} \hspace{1cm} \tilde{\p{r}}= (r_S, \dots , r_S,r_T).
\end{align*}
Let $r_u$ and $r_p$ denote the underlying spatial polynomial degrees for the displacement and pressure. Then the discrete displacement and pressure spaces  are 
\begin{align*}
\p{\mathcal{V}}_{h,r_u,r_T} &\coloneqq \big( \mathcal{V}_{h,r_u,r_T}\big)^d \cap \{ \p{v} \in \big(C^0(\mathcal{Q}) \big )^d \ |  \ \p{v}=0 \ \textup{on} \ \Sigma_u \cup \Sigma_0\} ,  \\ \ 
\mathcal{W}_{h,r_p,r_T} &\coloneqq  \mathcal{V}_{h,r_p,r_T}  \cap \{ q \in C^0(\mathcal{Q})   \ |  \ q=0 \ \textup{on} \ \Sigma_p \cup \Sigma_0 \} ,
\end{align*}
with $\Sigma_0 \coloneqq \Omega \times \{0 \}$.
We remark that it is possible to write these function spaces as product spaces, namely 
\begin{align}
\label{eq : product spaces}
&\p{\mathcal{V}}_{h,r_u,r_T} = \p{V}_{h_S,r_u} \otimes V_{h_T,r_T}   \hspace{1cm} \textup{and}  \hspace{1cm} \mathcal{W}_{h,r_p,r_T} = W_{h_S,r_p}  \otimes  V_{h_T,r_T}, 
\end{align}
\textup{where} $\p{V}_{h_S,r_u}, \ W_{h_S,r_p}$  \textup{are NURBS based approximation spaces corresponding to the spatial discretization and} $V_{h_T,r_T}  $ \textup{is the finite dimensional space for the time discretization}.  

\subsection{Discrete space-time variational formulation}
\label{subsec:discrete_formulation}
Starting point for the discretized variational formulation is the  classical Biot two-field model.
For the derivation we consider the solution to satisfy $p \in H^2(\mathcal{Q})$ and $\p{u} \in \p{H}^3(\mathcal{Q})$
while the right-hand sides fulfil $\p{f} \in \p{H}^1(\mathcal{Q}), \ g \in L^2(\mathcal{Q})$.  Since  the Biot equations define a PDAE  we combine the original Biot sytem with the differentiated first  equation 
\begin{align}
\label{eq:Biot_1_dt}
-h_T \, \partial_t\nabla_x \cdot  \p{\tilde{\sigma}}(\p{u},p)= h_T \, \partial_t\p{f}. 
\end{align}
This last differentiation step is inspired by the differentiation procedure used in  DAE (differential algebraic equation) theory in order to obtain  underlying ODEs (ordinary differential equations); see e.g. \cite{Simeon}.
To be more precise,  we choose a differentiated test function $ \partial_t \p{v}_h, \ \p{v}_h \in \p{\mathcal{V}}_{h,r_u,r_T}$ and multiply the sum \eqref{Eq1} $+ $ \eqref{eq:Biot_1_dt}  of the two equations by this test function. Integration over the whole space-time cylinder yields 
\begin{align*}
\langle	-\nabla_x \cdot  (\p{\sigma}- b \, p \p{I})- h_T \, \partial_t [\nabla_x \cdot ( \p{\sigma}- b \, p \p{I}) ] \ , \ \partial_t\p{v}_h \rangle_{\p{L}^2(\mathcal{Q})}=\langle \p{f} + h_T \partial_t \p{f} ,\partial_t\p{v}_h \rangle_{\p{L}^2(\mathcal{Q})},  
\end{align*}
and 
integration by parts along with the Einstein summation convention leads to
\begin{align}
\label{eq:Blueprint1}
&\int_{\mathcal{Q}} \ C_{ijkl} \, {\varepsilon}_{kl}(\p{u}+ h_T \partial_t \p{u}) \,  {\varepsilon}_{ij}(\partial_t\p{v}_h) \ d\p{x}dt- b \,  \int_{\mathcal{Q}} \,  (p + h_T \partial_t p)  \, \nabla_x  \cdot \partial_t\p{v}_h    \, d\p{x}dt   \\  
&\hspace{0.5cm}+\int_{\partial \mathcal{Q}} \, \big( \underbrace{(-{\p{\sigma}} + b \, p \p{I})}_{=-\p{\tilde{\sigma}}} + h_T \partial_t  (-{\p{\sigma}} + b \, p \p{I}) \big) \cdot \p{n}_x \  \partial_t\p{v}_h  \, ds	 = \langle \p{f} + h_T \partial_t \p{f} ,\partial_t\p{v}_h \rangle_{\p{L}^2(\mathcal{Q})}. \nonumber
\end{align}

On the other hand, the multiplication of the evolution equation \eqref{Eq2} for the pressure by a time-upwind test function $q_h + h_T \, \partial_t q_h$ gives, again using integration by parts,
\begin{align}
\label{eq:Blueprint2}
&\int_{\mathcal{Q}} \, c_0 \, \partial_t p \ (q_h+h_T \partial_t q_h)  \, d\p{x}dt + b \, \int_{\mathcal{Q}} \, \nabla_x \cdot \partial_t\p{u}  \ (q_h+h_T \partial_t q_h)  \, d\p{x}dt   \nonumber \\ & \hspace{0.3cm}-\int_{\partial \mathcal{Q}} \big(\p{\mathcal{K}}\nabla_x p \, (q_h + \partial_t q_h) \big) \cdot \p{n}_x \,ds+ \int_{\mathcal{Q}} \, \p{\mathcal{K}}\nabla_xp \  \nabla_x  (q_h + h_T \, \partial_tq_h)  \, d\p{x}dt  \\  &\hspace{5cm}  =\langle g  ,  q_h + h_T \partial_t q_h\rangle_{L^2(\mathcal{Q})}. \nonumber
\end{align}

\begin{remark}
	For the derivation above, the product structure of the parametrization
	and the smoothness of the isogeometric basis functions in each mesh element $\mathcal{K}$ lead to the well-definedness of the mixed derivatives $\partial_{x_i} \partial_t\p{v}_h, \ \partial_{x_i}\partial_t q_h$. The $\partial_{x_i}, \ i = 1, \dots , d$\ , denote the derivatives w.r.t. spatial coordinates.
\end{remark}

Both equations \eqref{eq:Blueprint1} and \eqref{eq:Blueprint2} are  in some sense the blueprints  for the next definition.
\begin{definition}{(Discrete variational formulation)}\\
	\label{def:discrete variational form}
	Find ${u}_h \in \p{\mathcal{V}}_{h,r_u,r_T}$ and $p_h \in \mathcal{W}_{h,r_p,r_T}$ s.t.
	\begin{align}
	\label{Biot_Weak_Eq1}
	\tilde{e}(\p{u}_h+h_T \, {\partial_t\p{u}}_h  \, , \, {\partial_t\p{v}}_h)- b\langle p_h+ h_T \,  \partial_t{p}_h  \, , \,  \nabla_x \cdot {\partial_t\p{v}}_h\rangle_{L^2(\mathcal{Q})}=l_1({\partial_t\p{v}}_h ), 
	\\ 	\label{Biot_Weak_Eq2}
	c_0 \langle \partial_t{p}_h \, , \,  q_h+ h_T \, \partial_t{q}_h \rangle_{L^2(\mathcal{Q})}+ b\langle \nabla_x \cdot {\partial_t\p{u}}_h  \, , \, q_h+ h_T \, \partial_t{q}_h \rangle_{L^2(\mathcal{Q})} + \tilde{a}(p_h \, ,\,  q_h+ h_T \, \partial_t{q}_h )\nonumber \\ =  l_2 (q_h+ h_T \, \partial_t{q}_h) ,
	\end{align}
	\normalsize
	for all $ {v}_h \in \p{{\mathcal{V}}}_{h,r_u,r_T}$ and $q_h\in {\mathcal{W}}_{h,r_p,r_T}$, \vspace{0.1cm} \\
	\begin{minipage}{3cm}
		with linear  forms
	\end{minipage}
	\begin{minipage}{12.7cm}
		$\begin{cases}
		l_1(\p{v}) \coloneqq \langle \p{f} + h_T  \partial_t\p{f} \, , \, \p{v} \rangle_{\p{L}^2(\mathcal{Q})}+ \langle \p{t}_n + h_T  \partial_t\p{t}_n \, , \, \p{v} \rangle_{\p{L}^2(\Sigma_t)},\vspace{0.1 cm} \\
		l_2(q ) \coloneqq \langle  g  \, , \, q \rangle_{L^2(\mathcal{Q})} +  \langle  v_f  \, , \, q \rangle_{L^2(\Sigma_f)}, 
		\end{cases}$
	\end{minipage}
    \vspace{0.2cm}\\
	\begin{minipage}{3cm}
		and bilinear forms
	\end{minipage}
	\begin{minipage}{9.5cm}
		$\begin{cases}
		\tilde{e}(\p{u},\p{v}) \coloneqq  \int_{\mathcal{Q}} C_{ijkl} {\varepsilon}_{kl}(\p{u}) {\varepsilon}_{ij}(\p{v}) \ d\p{x}dt,  \vspace{0.1 cm} \\
		\tilde{a}(p,q) \coloneqq   \langle \p{\mathcal{K}}\nabla_x p \, , \, \nabla_xq \rangle_{\p{L}^2(\mathcal{Q})} . 
		\end{cases}$
	\end{minipage}
\end{definition} 

For later considerations, we use instead of \eqref{Biot_Weak_Eq1}-\eqref{Biot_Weak_Eq2}  the equivalent formulation 
\begin{align}
\label{disctre_variational_problem_short}
\textup{Find}  \ \p{u}_h, p_h  \ \ \textup{s.t.}  \ \  b_{ST}([\p{u}_h,&p_h],[\p{v}_h,q_h])=l([\p{v}_h,q_h]), \ \  \forall \  \p{v}_h \in \p{\mathcal{V}}_{h,r_u,r_T},  \ q_h \in \mathcal{W}_{h,r_p,r_T},  \\ \textup{with}  \hspace{0.3cm}
\label{eq:discrete_bilinear}
b_{ST}([\p{u},p],[\p{v},q]) &\coloneqq 	\tilde{e}(\p{u}+h_T \, {\partial_t\p{u}}  \, , \, {\partial_t\p{v}})- b\langle p+ h_T \,  \partial_t{p}  \, , \,  \nabla_x \cdot {\partial_t\p{v}}\rangle_{L^2(\mathcal{Q})}  \\ & \hspace{0.15cm}+ 	c_0 \langle \partial_t{p} \, , \,  q+ h_T \, \partial_t{q} \rangle_{L^2(\mathcal{Q})}+ b\langle \nabla_x \cdot {\partial_t\p{u}}  \, , \, q+ h_T \, \partial_t{q} \rangle_{L^2(\mathcal{Q})} \nonumber \\ & \hspace{0.35cm} + \tilde{a}(p \, ,\,  q+ h_T \, \partial_t{q}), \nonumber \\
l([\p{v},q]) &\coloneqq l_1({\partial_t\p{v}} ) +  l_2 (q+ h_T \, \partial_t{q}). \nonumber
\end{align} 
The discrete formulation is consistent in the following way.
\begin{lemma}
	\label{lemma:consistency}
	Assume that there exists a solution to \eqref{Eq1}-\eqref{Eq2} satisfying  $\p{u} \in \p{H}^3(\mathcal{Q})$ and $p \in H^2(\mathcal{Q})$ and let $v_f \in L^2(\Sigma_f)$ and $\p{t}_n \in \p{L}^2(\Sigma_t)$ the restriction of a function in $\p{t}_e \in \p{H}^1(\partial\mathcal{Q})$. Moreover, let $\p{f}\in \p{H}^1(\mathcal{Q}), \ g \in L^2(\mathcal{Q})$.\\ Then  $b_{ST}([\p{u},p],[\p{v}_h,q_h])=l([\p{v}_h,q_h]) \ \ \textup{for all} \  \p{v}_h \in \p{\mathcal{V}}_{h,r_u,r_T}, \, q_h \in \mathcal{W}_{h,r_p,r_T}$. 
\end{lemma}
\begin{proof}
	This is clear due to the equations \eqref{eq:Blueprint1} and \eqref{eq:Blueprint2}, the assumed boundary conditions and the fact that for the Lipschitz domain $\mathcal{Q}$ the trace operator restricted to $H^2(\mathcal{Q})$ defines a linear and continuous operator $\gamma_0 \colon H^2(\mathcal{Q}) \rightarrow H^1(\partial\mathcal{Q})$; see, e.g., \cite{geymonat}.
\end{proof}

Note the well-definedness of the terms on the right-hand side of \eqref{eq:discrete_bilinear} if   $\p{u}, \, \p{v} \in \p{\mathcal{V}}_0$ and $p,q \in \mathcal{W}_0$, where
\begin{align*}
\p{\mathcal{V}}_0  &\coloneqq \{ \p{v}=(v_1, \dots ,v_d) \in \p{H}^1(\mathcal{Q})\ | \    \nabla_x\partial_t v_i \in \p{L}^2(\mathcal{Q}), \  v_i=0 \, \textup{ on } \ \Sigma_u \cup \Sigma_0 , \ \forall \, i \}, \vspace{0.001cm} \nonumber \\  \mathcal{W}_0 &\coloneqq \{ q \in H^1(\mathcal{Q}) \ |  \  \nabla_x\partial_t q \in \p{L}^2(\mathcal{Q}), \ q=0 \ \textup{on} \ \Sigma_p \cup \Sigma_0 \}.
\end{align*}

\subsection{Existence of solutions}
\label{subsec:existence_discrete_solution}
Next we check if there exists a solution to the discrete problem. 
For this purpose we prove   the coercivity of the bilinear form $b_{ST}(\cdot , \cdot )$ w.r.t. to the space $\p{\mathcal{V}}_{h,r_u,r_T} \times \mathcal{W}_{h,r_p,r_T}$ endowed with the auxiliary norm 
\begin{align}
\label{hnorm}
&\norm{[\p{v},q]}_h^2 \coloneqq h_T \norm{\partial_t\p{v}}_{\p{\mathcal{H}}^1(\mathcal{Q})}^2 +\norm{\p{v}}_{\p{H}^1(\Sigma_T)}^2+h_T  c_0  \norm{\partial_t q}_{L^2(\mathcal{Q})}^2 \\ & \hspace{5cm} +c_0  \norm{q}_{L^2(\Sigma_T)}^2+\norm{\nabla_xq}_{\p{L}^2(\mathcal{Q})}^2,    \hspace{0.5cm}  \textup{where} \nonumber \\ & \norm{\p{v}}_{\p{\mathcal{H}}^1(\mathcal{Q})}^2\coloneqq \int_{\mathcal{Q}} \sum_{i,j} \big( \partial_{x_i} v_j \big)^2 \, d\p{x}dt+ \norm{\p{v}}_{\p{L}^2(\mathcal{Q})}^2, \ \ \p{v}=(v_1,\dots,v_d)   \   \ \textup{and} \  \ \Sigma_T = \Omega \times \{T \}.\nonumber
\end{align}
Exploiting the continuity, the piecewise smoothness and the boundary conditions of the test functions one can check easily that $\norm{[\cdot,\cdot]}_h$ is indeed a norm in $\p{\mathcal{V}}_{h,r_u,r_T} \times \mathcal{W}_{h,r_p,r_T}$. Before we prove the coercivity we insert here two auxiliary results. 
\begin{lemma}
	\label{theorem Korn}
	Let the elasticity tensor $\p{C}$ satisfy $$ \mu(\p{C}) \ \sum_{i,j} \, x_{ij}^2 \leq \sum_{i,j,k,l}C_{ijkl} \, x_{ij}{x}_{kl}$$ for all $x_{ij}, \ x_{kl} \in \mathbb{R}, \ i,j,k,l \in \{1,\dots,d\}$ and some constant $0<\mu(\p{C})$. 
	Then the bilinear  form  
	\begin{align*}
	e(\cdot,\cdot) \colon \p{V} \times \p{V} \rightarrow \mathbb{R} \ , \ (\p{u},\p{v}) \mapsto   \,\int_{\Omega}  C_{ijkl} {\varepsilon}_{kl}(\p{u}) {\varepsilon}_{ij}(\p{v})  \, d\p{x}
	\end{align*} is symmetric and coercive w.r.t. the $\p{H}^1$-norm in  $\p{V} \coloneqq \{ \p{v} \in \p{H}^1(\Omega) \ | \ \p{v}=0 \ \textup{on} \ \Gamma_u \}$.
	Thus there exists a constant $0<c_{e}$ such that 
	\begin{align}
	\label{eq:Korn_inequality}
	e(\p{v},\p{v}) \geq c_{e} \norm{\p{v}}_{\p{H}^1(\Omega)}^2, \ \ \ \ \forall \ \p{v} \in \p{V}.
	\end{align}
	The constant depends only on $\Omega, \ \Gamma_u $ and $\p{C}$.
\end{lemma}
\begin{proof}
	The inequality \eqref{eq:Korn_inequality} follows  by  \cite{Alessandrini2006TheLC}, Corollary 5.9, for both cases $d=2,3$. The one-dimensional  case is a consequence of  the  Poincar\'{e} inequality; see Example 3  in \cite{grser2015note}. And the symmetry of $e(\cdot ,\cdot)$ is clear due to the symmetry properties of the elasticity tensor.
\end{proof}

\begin{lemma}
	\label{ses_a}
	There exist constants $0<c_a,C_a<\infty$ which only depend on $\p{\mathcal{K}}, \ \Gamma_p$ and $\Omega$ such that
	\begin{align*}
	a(q,q) \geq c_a \, \norm{q}_{H^1(\Omega)}^2 \ \ \ \textup{and} \ \ \ |a(p,q)| \leq  C_a \, \norm{p}_{H^1(\Omega)} \norm{q}_{H^1(\Omega)},  \\  \  \textup{for all} \  \ \ p, \, q \in W \coloneqq \{ w \in H^1(\Omega) \ | \ w = 0 \ \textup{on} \  \Gamma_p \} \ \ 	\textup{and with} \nonumber \\
	\ \ \ a \colon W \times W \rightarrow \mathbb{R} \ , \  (p,q) \mapsto  \langle \p{\mathcal{K}}\nabla_x p \, , \, \nabla_xq \rangle_{\p{L}^2(\Omega)}.   \ \ \ \  \ \ \ \ \ \ \ \nonumber
	\end{align*}
\end{lemma} 
\begin{proof}
	An application of the  Poincar\'{e} inequality and the assumption that the $ \p{\mathcal{K}}(\p{x}) $   are uniformly elliptic and bounded symmetric positive definite matrices give the assertion.
\end{proof}

\begin{remark}
	\label{remark:derivatives_of _test_func}
	In the context of the space-time discretization we interpret the derivatives $\partial_{x_i}, \ \partial_t$ etc., as weak derivatives w.r.t. to the domain $\mathcal{Q}$. Nevertheless due to the piecewise smoothness of the test functions and their continuity, we obtain their weak derivatives as piecewise defined classical derivatives and in particular one can assume  for $s \in [0,T]$ that $\p{v}_h(\cdot , s) \in \p{V}$ and $q_h(\cdot , s) \in W$. The product structure of the parametrization  and hence of the test spaces further gives us  $\partial_t\p{v}_h(\cdot , s) \in \p{V}$ and $\partial_tq_h(\cdot , s) \in W$; see \eqref{eq : product spaces}.
\end{remark}

Now we arrive at the mentioned coercivity result.
\begin{lemma}
	\label{lemma:coercivity}
	The bilinear form $b_{ST}$, defined  by \eqref{eq:discrete_bilinear}, is coercive in the sense that there exists a constant $0 < \mu_c$ independent from the mesh sizes such that
	\begin{align*}
	\mu_c \, \norm{[\p{u}_h,p_h]}_h^2 \leq b_{ST}([\p{u}_h,p_h],[\p{u}_h,p_h]), \ \ \forall \ [\p{u}_h,p_h] \in  \p{\mathcal{V}}_{h,r_u,r_T} \times \mathcal{W}_{h,r_p,r_T}.
	\end{align*}
	The constant $\mu_c$ can be chosen independently from $c_0$.
\end{lemma}
\begin{proof}
	For reasons of clarity we estimate  different terms in the definition of $b_{ST}$, i.e., \eqref{eq:discrete_bilinear}, separately, starting with the non-mixed terms in which either only  $\p{u}$ and $\p{v}$ or only $p, \ q$ occur.
	\begin{itemize}
		\item Consider the $\tilde{e}(\cdot, \cdot)$ term and note  that $\p{u}_h(\cdot,t=0)=\p{0}$. Observe the symmetry of the bilinear form $\tilde{e}(\cdot,\cdot),$ that is obvious from the symmetry of $e(\cdot,\cdot)$ ; see Lemma \ref{theorem Korn}. By means of these properties and Green's formula  we can write 
		\begin{align*}
		\tilde{e}(\p{u}_h,\partial_t\p{u}_h) &=  \int_{\mathcal{Q}} C_{ijkl} {\varepsilon}_{kl}(\p{u}_h) {\varepsilon}_{ij}(\partial_t\p{u}_h)d\p{x} dt  = \\
		&= \frac{1}{2} \int_{\mathcal{Q}}  \partial_t \big(C_{ijkl} {\varepsilon}_{kl}(\p{u}_h) {\varepsilon}_{ij}(\p{u}_h) \big) \, dt d\p{x} \\ &= \frac{1}{2}  \int_{\partial \mathcal{Q}}  C_{ijkl} {\varepsilon}_{kl}(\p{u}_h) {\varepsilon}_{ij}(\p{u}_h) \cdot n_t \, ds \\
		&= \frac{1}{2}  \,\int_{L^2(\Sigma_T)}  C_{ijkl} {\varepsilon}_{kl}(\p{u}_h) {\varepsilon}_{ij}(\p{u}_h)  \, d\p{x} \\ 
		& = \frac{1}{2} \, e(\p{u}_h(\cdot ,T),\p{u}_h( \cdot ,T)) \geq \frac{c_e}{2 } \, \norm{\p{u}_h(\cdot, T)}_{\p{H}^1(\Omega)}^2. 
		\end{align*}
		We used here the piecewise smoothness of the test functions, i.e., $\p{u}_h(\cdot ,t) \in \p{V}$, and the coercivity of the elasticity form $e(\cdot,\cdot)$ (Lemma \ref{theorem Korn} ) with $(\p{n}_x,n_t)$ denoting the outer unit normal vector of the space-time domain.
		\item In view of  Remark \ref{remark:derivatives_of _test_func} and the coercivity of $e(\cdot, \cdot)$  it is
		\begin{align*}
		h_T \, \tilde{e}(\partial_t\p{u}_h,\partial_t\p{u}_h) &= h_T \, \int_{\mathcal{Q}} C_{ijkl} {\varepsilon}_{kl}(\partial_t\p{u}_h) {\varepsilon}_{ij}(\partial_t\p{u}_h)d\p{x}  dt \\
		& =  h_T \, \int_I {e}(\partial_t\p{u}_h(\cdot ,t) , \partial_t\p{u}_h(\cdot ,t)) \, dt \\ & \geq h_T \, \int_I \, c_e \norm{ \partial_t \p{u}_h(\cdot ,t)}_{\p{H}^1(\Omega)}^2 \, dt \\
		& = h_T \, c_e \, \int_I \int_{\Omega} \sum_{i,j} (\partial_{x_i}\partial_t u_{h,j} )^2+ \sum_j (\partial_t u_{h,j})^2 \, d\p{x} dt \\ 
		&=h_T \, c_e \,  \norm{\partial_t\p{u}_h}_{\p{\mathcal{H}}^1(\mathcal{Q})}^2.
		\end{align*}		
		One notices Fubini's theorem and the notation $u_{h,i}$ for the i-th component of $\p{u}_h$.
		\item By the chain rule and the zero initial conditions we get
		\begin{align*}
		\langle c_0 \, \partial_t p_h, p_h \rangle_{L^2(\mathcal{Q})} &= \int_{\mathcal{Q}}  \, \frac{c_0}{2 } \partial_t \big( p_h^2 \big) dt d\p{x}= \frac{c_0}{2 } \int_{\partial \mathcal{Q}} \, p_h^2 \cdot n_t \, ds \\
		& =\frac{c_0}{2 } \, \norm{p_h}_{L^2(\Sigma_T)}^2.
		\end{align*}
		\item  Obviously,
		\begin{align*}
		h_T \, \langle c_0 \, \partial_t p_h , \partial_t p_h \rangle_{L^2(\mathcal{Q})} = h_T \,  c_0 \, \norm{\partial_t p_h}_{L^2(\mathcal{Q})}^2.
		\end{align*}
		\item Moreover,  with Remark \ \ref{remark:derivatives_of _test_func} as well as with Lemma \ref{ses_a} and Fubini's theorem one can estimate
		\begin{align*}
		\tilde{a} ( p_h,p_h) = \int_{\Omega}  \int_I \, \p{\mathcal{K}} \nabla_x p_h \  \nabla_x  p_h  \, dtd\p{x} =   \int_I \,a(p_h(\cdot , t),p_h(\cdot , t))  \, dt  \\ 
		\geq c_a \norm{\nabla_xp_h}_{\p{L}^2(\mathcal{Q})}^2.
		\end{align*}
		\item Finally the last non-mixed term yields by the symmetry of $a(\cdot , \cdot)$:
		\begin{align*}
		h_T \, \tilde{a} (p_h,\partial_t p_h)&=h_T \, \int_I  a(p_h(\cdot , t),\partial_t p_h(\cdot , t)) \, dt
		\\ & =   \frac{h_T}{2}\int_{I} \partial_t a(p_h(\cdot , t),p_h(\cdot , t)) \, dt  = \frac{h_T}{2} \,  a(p_h(\cdot ,T),p_h(\cdot ,T)) \geq 0  .
		\end{align*}
		
	\end{itemize}
	For the first, third and last point above, we used the assumption  that $p_h=0, \p{u}_h=0$ on $\Sigma_0$. 
	Next we sum up all the remaining terms in the definition of $b_{ST}([\p{u}_h,p_h],[\p{u}_h,p_h])$, i.e. in the sum of the right-hand side of \eqref{eq:discrete_bilinear}. We get
	\begin{align*}
	& b \ \big[- \langle p_h,\nabla_x \cdot \partial_t \p{u}_h\rangle_{L^2(\mathcal{Q})}- h_T \, \langle \partial_t p_h,\nabla_x \cdot \partial_t\p{u}_h\rangle_{L^2(\mathcal{Q})} \\ 
	&\hspace{2cm}+ \langle \nabla_x \cdot \partial_t\p{u}_h , p_h \rangle_{L^2(\mathcal{Q})} + h_T \, \langle\nabla_x \cdot \partial_t\p{u}_h , \partial_t p_h \rangle_{L^2(\mathcal{Q})} \big]=0.
	\end{align*}
	Thus the mixed terms vanish. So it is obvious by the above estimates that
	\begin{align*}
	b_{ST}([\p{u}_h,p_h],[\p{u}_h,p_h])&  
	\geq  \mu_c \, \Big( \norm{\p{u}_h}_{\p{H}^1(\Sigma_T)}^2 + h_T \,  \norm{ \partial_t \p{u}_h}_{\p{\mathcal{H}}^1(\mathcal{Q})}^2  
	 +c_0 \, \norm{p_h}_{L^2(\Sigma_T)}^2  \\ & \hspace{1.8cm} + h_T \, c_0 \,  \norm{\partial_t p_h}_{L^2(\mathcal{Q})}^2+ \norm{\nabla_xp_h}_{\p{L}^2(\mathcal{Q})}^2\Big)
	\\
	&= \mu_c  \, \norm{[\p{u}_h,p_h]}_h^2,
	\end{align*}
	for $\mu_c \coloneqq \min\{ \frac{c_e}{2} , c_a, \frac{1}{2} \}.$
\end{proof}

If we now choose  bases $\{ \psi_i \ | \ i=1,\dots , N^p_h \}$ and  $\{ \p{\phi}_j \ | \ j=1,\dots , N^{u}_h \}$ of the test spaces $\mathcal{W}_{h,r_p,r_T}$, $\p{\mathcal{V}}_{h,r_u,r_T}$ respectively. Then the  coefficient vectors $P^h \coloneqq (\hat{p}_{1}, \dots, \hat{p}_{N^p_h})$, $U^h \coloneqq (\hat{u}_{1}, \dots  ,\hat{u}_{N^{u}_h})$, with 
\begin{align*}
&\p{u} \approx \p{u}_h \coloneqq \sum_{j=1}^{N^{u}_h} \hat{u}_{j} \, \p{\phi}_j,  &p \approx p_h \coloneqq \sum_{i=1}^{N^p_h} \hat{p}_{i} \, \psi_i,
\end{align*}
which define the discretized solution $\p{u}_h,p_h$ are obtained by solving one  linear system of the type
\begin{align*}
S^h			
\begin{bmatrix}
U^h \\ P^h
\end{bmatrix}=\begin{bmatrix}
R^{1} \\ R^2
\end{bmatrix}.                                                                         
\end{align*} 	

The shown coercivity of the bilinear form $b_{ST}$  implies the  positive definiteness  of the  system matrix $S^h$ .
Thus the existence of a unique solution is clear. 

\begin{theorem}{\textup{\textbf{Existence of a discretized solution}}}\\
	There exists a unique solution to the variational problem \eqref{disctre_variational_problem_short}.
\end{theorem}

Latter theorem guarantees the well-definedness of our numerical scheme. But for a useful method a convergence statement is an important aspect, too. Consequently we face this issue in the next part.


\section{Error analysis}
\label{sec:error_analysis}
The main objective of this section is the derivation of an error estimate for the numerical approximation of 
the displacement $\p{u}$ and the pressure $p$ in the setting of Lemma \ref{lemma:consistency}. For reasons of simplification we set in the whole chapter w.l.o.g. $b=1$ and remark that $0.5 \leq b  \leq 1$ in most applications.\\
We start with a result from the IGA theory that will be used below.  

\begin{lemma}{(Inverse inequality)}	\label{inversineq_ad} \\
	Let the space-time mesh be regular with polynomial degrees  $  r_T, \,  r_S$ greater than zero. Then for $i=1, \dots,d$
	it holds  
	\begin{align}
	\label{eq1}
	h_T \, \norm{\partial_t \partial_{x_i}v_h}_{L^2(\mathcal{Q})} &\leq C_{inv,1} \, \norm{\partial_{x_i} v_h}_{L^2(\mathcal{Q})}, \\ 	\label{eq2}
	h_T \, \norm{\partial_tv_h}_{L^2(\mathcal{Q})} &\leq C_{inv,2} \, \norm{v_h}_{L^2(\mathcal{Q})}, \ \ v_h \in \mathcal{V}_{h, r_S,r_T} \cap C^0(\mathcal{Q}), 
	\end{align}
	where $C_{inv,j}$ are constants independent of the mesh sizes and $v_h$.
\end{lemma}
\begin{proof}
	We remark that $\partial_{x_i} v_h$ is piecewise smooth and for  $\p{x} \in \mathcal{Q}$ the function  $\partial_{x_i} v_h(\p{x}, \cdot)$ is continuous in time. By the product structure of the space $\mathcal{V}_{h,\tilde{\p{r}}}$ and $\mathcal{Q}$ one sees that $\partial_{x_i} v_h(\p{x}, \cdot)$ is an element of a univariate spline space $ V_{h_T,r_T}$ with mesh size $h_T$. Due to the regularity of the mesh and as a consequence of Theorem 4.2 in \cite{IGA3} we find a  constant $C$ independent of $\p{x}$ s.t. the estimate
	$$ h_T \norm{\partial_t\partial_{x_i} v_h(\p{x}, \cdot)}_{L^2(I)} \leq C \norm{\partial_{x_i} v_h(\p{x}, \cdot)}_{L^2(I)}.$$ is fulfilled. Integration over $\Omega$ yields the assertion for inequality \eqref{eq1}. The second estimate \eqref{eq2} can be proven analogously.
\end{proof}

Next we define for  the space $\p{\mathcal{V}}_0 \times \mathcal{W}_0 $ another auxiliary norm 
\begin{align}
\label{def:h_star_norm}
\norm{[\p{u},p]}_{h,\star}^2 \coloneqq     \norm{\partial_t\p{u}}_{\p{\mathcal{H}}^1(\mathcal{Q})}^2 + \frac{1}{h_T}   \big( \norm{\p{u}}_{\p{\mathcal{H}}^1(\mathcal{Q})}^2 + \norm{p}_{L^2(\mathcal{Q})}^2\big) + \norm{\partial_t p}_{L^2(\mathcal{Q})}^2+\norm{\nabla_xp}_{\p{L}^2(\mathcal{Q})}^2
\end{align}
and state a boundedness result for $b_{ST}$.
\begin{lemma}
	\label{lemma:continuity_mathfrak_A}
	\label{lemma:continuityA}
	The bilinear form $b_{ST}$ is continuous w.r.t. the norms $\norm{\cdot}_h$ and $\norm{\cdot}_{h,\star}$ in the sense $$b_{ST}([\p{u},p],[\p{u}_h,p_h]) \leq \mu_b \, \norm{[\p{u},p]}_{h,\star}  \norm{[\p{u}_h,p_h]}_h$$
	for all $[\p{u},p] \in \p{\mathcal{V}}_0 \times \mathcal{W}_0$ and $[\p{u}_h,p_h] \in \p{\mathcal{V}}_{h,r_u,r_T} \times \mathcal{W}_{h,r_p,r_T}$ and  some constant $\mu_b$ independent of $c_0$ and the mesh size.
\end{lemma}
\begin{proof}
	We first look at the different terms appearing in the definition of $b_{ST}$ and estimate them separately.  Doing so, we also introduce some auxiliary constants $C_1, \dots, C_8$.
	\begin{itemize}
		\item By the definition of the elasticity bilinear form $\tilde{e}(\cdot,\cdot)$ (see Definition \ref{def:discrete variational form}) and the Cauchy-Schwarz inequality  we have: 
		\begin{align*}
		S_1(\p{u},p,\p{u}_h,p_h) &\coloneqq \tilde{e}(\p{u},\partial_t\p{u}_h) 
		= \int_{\mathcal{Q}} C_{ijkl} \p{\epsilon}_{kl}(\p{u}) \p{\epsilon}_{ij}(\partial_t\p{u}_h) \, d\p{x}dt \\	& =  \int_{\mathcal{Q}}    \frac{1}{4} \, C_{ijkl} \big( \partial_{x_k} u_l \, \partial_{x_i} \partial_t u_{h,j}   +\partial_{x_l} u_k \, \partial_{x_i}  \partial_t u_{h,j} \\ & \hspace{3cm} +\partial_{x_k} u_l \, \partial_{x_j} \partial_t u_{h,i}+\partial_{x_l}  u_k \, \partial_{x_j} \partial_t u_{h,i} \  \big) d\p{x}dt \\
		& \leq  C_{\tilde{e}}  \frac{1}{\sqrt{h_T}}\norm{\p{u}}_{\p{\mathcal{H}}^1(\mathcal{Q})} \sqrt{h_T}\norm{\partial_t\p{u}_h}_{\p{\mathcal{H}}^1(\mathcal{Q})}     \\
		& \leq C_{1} \, \norm{[\p{u},p]}_{h,\star} \norm{[\p{u}_h,p_h]}_{h}.
		\end{align*}
		Above we can set $	C_{\tilde{e}} \coloneqq \sum_{i,j,k,l} |C_{ijkl}| $ and $u_{i} $ denotes the $i$-th component of $\p{u}$.
		\item  The Cauchy-Schwarz inequality  yields 
		\begin{align*}
		S_2(\p{u},p,\p{u}_h,p_h) & \coloneqq -\langle p,\nabla_x \cdot \partial_t \p{u}_h\rangle_{L^2(\mathcal{Q})} 
		\leq  \norm{p}_{L^2(\mathcal{Q})}  \norm{\nabla_x \cdot \partial_t\p{u}_h}_{L^2(\mathcal{Q})}   \\
		& \leq \frac{1}{\sqrt{h_T}}\norm{p}_{L^2(\mathcal{Q})}   \sqrt{3 \, h_T}\norm{\partial_t\p{u}_h}_{\p{\mathcal{H}}^1(\mathcal{Q})}  \\
		&\leq C_2 \norm{[\p{u},p]}_{h,\star}  \norm{[\p{u}_h,p_h]}_{h}.
		\end{align*} 
		\item In an analogous manner to the first point  it holds 
		\begin{align*}
		S_3(\p{u},p,\p{u}_h,p_h)  \coloneqq h_T \, \tilde{e}( \partial_t \p{u},\partial_t\p{u}_h) & \leq C_{\tilde{e}} \, h_T \, \norm{\partial_t\p{u}}_{\p{\mathcal{H}}^1(\mathcal{Q})} \norm{\partial_t\p{u}_h}_{\p{\mathcal{H}}^1(\mathcal{Q})} \\
		&  \leq  C_{\tilde{e}} \, \sqrt{h_T} \, \norm{\partial_t\p{u}}_{\p{\mathcal{H}}^1(\mathcal{Q})} \sqrt{h_T} \, \norm{\partial_t\p{u}_h}_{\p{\mathcal{H}}^1(\mathcal{Q})} \\
		& \leq  C_{3}  \, \norm{[\p{u},p]}_{h,\star}  \norm{[\p{u}_h,p_h]}_{h}. 
		\end{align*}
		\item A further  term can be bounded  similarly as in point 2:
		\begin{align*}
		S_4(\p{u},p,\p{u}_h,p_h) & \coloneqq -h_T \, \langle \partial_t p,\nabla_x \cdot\partial_t\p{u}_h\rangle_{L^2(\mathcal{Q})} \\
		&\leq  h_T \, \norm{\partial_t p}_{L^2(\mathcal{Q})}   \norm{\nabla_x \cdot  \partial_t\p{u}_h}_{L^2(\mathcal{Q})} \\ & \leq  \sqrt{h_T} \norm{\partial_t p}_{L^2(\mathcal{Q})}  \sqrt{3\,h_T}   \,\norm{\partial_t\p{u}_h}_{\p{\mathcal{H}}^1(\mathcal{Q})} \\
		&\leq C_4 \norm{[\p{u},p]}_{h,\star}  \norm{[\p{u}_h,p_h]}_{h}.
		\end{align*}
		\item Then we have by the assumption $c_0\leq 1$ and by means of the Poincar\'{e} inequality  (see, e.g., Example 3  in \cite{grser2015note}) for some constant $C_P$:
		\begin{align*}
		S_5(\p{u},p,\p{u}_h,p_h) & \coloneqq 
		\langle c_0 \, \partial_t p+\nabla_x \cdot \partial_t\p{u} , p_h \rangle_{L^2(\mathcal{Q})}   \\
		& \leq ( c_0 \norm{\partial_t p}_{L^2(\mathcal{Q})} + \sqrt{3} \norm{ \partial_t\p{u}}_{\p{\mathcal{H}}^1(\mathcal{Q})} )  \,  \norm{p_h}_{L^2(\mathcal{Q})} \\
		& \leq  C_P \, ( c_0 \norm{\partial_t p}_{L^2(\mathcal{Q})} + \sqrt{3}\norm{ \partial_t\p{u}}_{\p{\mathcal{H}}^1(\mathcal{Q})} )   \, \norm{\nabla_xp_h}_{\p{L}^2(\mathcal{Q})}  \\
		& \leq C_5 \, \norm{[\p{u},p]}_{h,\star}  \norm{[\p{u}_h,p_h]}_{h}. 
		\end{align*}
		
		\item Since $\p{\mathcal{K}}$ is a symmetric positive-definite matrix one further obtains by the Definition \ref{def:discrete variational form}  that
		\begin{align*}
		S_6(\p{u},p,\p{u}_h,p_h) \coloneqq\tilde{a}(p,p_h) &\leq   \mu_K \,  \norm{\nabla_xp}_{\p{L}^2(\mathcal{Q})} \norm{\nabla_xp_h}_{\p{L}^2(\mathcal{Q})}   \\
		& \leq C_6 \norm{[\p{u},p]}_{h,\star}  \norm{[\p{u}_h,p_h]}_{h},
		\end{align*}
		where the positive number $\mu_K < \infty $ is the supremum over all  eigenvalues of the matrices $\p{\mathcal{K}}(\p{x})$.
		\item We proceed with the seventh term. Again Cauchy-Schwarz along with \eqref{eq2} yields
		\begin{align*}
		S_7(\p{u},p,\p{u}_h,p_h)& \coloneqq 
		h_T \, \langle c_0 \, \partial_t p+\nabla_x \cdot \partial_t\p{u} , \partial_t p_h \rangle_{L^2(\mathcal{Q})} \\ &\leq c_0 \, \sqrt{h_T} \norm{\partial_t p}_{L^2(\mathcal{Q})}  \sqrt{h_T} \norm{\partial_t p_h}_{L^2(\mathcal{Q})} \\ 
		& \hspace{2cm}+  \norm{ \nabla_x \cdot \,\partial_t\p{u}}_{L^2(\mathcal{Q})} h_T \, \norm{\partial_t p_h}_{L^2(\mathcal{Q})} \\
		& 
		\leq    \sqrt{c_0 \,h_T} \norm{\partial_t p}_{L^2(\mathcal{Q})} \sqrt{c_0\,h_T} \norm{\partial_t p_h}_{L^2(\mathcal{Q})}\\ 
		& \hspace{2cm} +\sqrt{3} \norm{ \partial_t\p{u}}_{\p{\mathcal{H}}^1(\mathcal{Q})}  \, C_{inv,2} \, \norm{p_h}_{L^2(\mathcal{Q})} \\  &\leq C_7 \norm{[\p{u},p]}_{h,\star}  \norm{[\p{u}_h,p_h]}_{h}.
		\end{align*}
		Note, to obtain the last inequality sign we applied again the Poincar\'{e} inequality\\ $\norm{p_h}_{L^2(\mathcal{Q})} \leq C_P \norm{\nabla_xp_h}_{\p{L}^2(\mathcal{Q})}$.
		\item Finally, for the last term and with $\mu_K$ defined in point 6 it is
		\begin{align*}
		S_8(\p{u},p,\p{u}_h,p_h) & \coloneqq 
		h_T	 \, \tilde{a}( p, \partial_t p_h )  \\ 
		&\leq h_T \,  \mu_K \, \norm{\nabla_xp}_{\p{L}^2(\mathcal{Q})}  \norm{\nabla_x \partial_t p_h}_{\p{L}^2(\mathcal{Q})}  \\  & \leq \mu_K  \, C_{inv,1} \,  \norm{\nabla_xp}_{\p{L}^2(\mathcal{Q})} \norm{\nabla_xp_h}_{\p{L}^2(\mathcal{Q})} \\
		&\leq  C_{8} \, \norm{[\p{u},p]}_{h,\star}  \norm{[\p{u}_h,p_h]}_{h}.
		\end{align*}
		Here the fore-last inequality sign follows by using the adapted inverse  estimate \eqref{eq2} in Lemma \ref{inversineq_ad}  and the regular  mesh assumption.
	\end{itemize}
	Summarizing,  we obtain the original form $b_{ST}$ as the sum of the different terms $S_i$ , i.e. $$\sum_{i=1}^{8} S_i(\p{u},p,\p{u}_h,p_h) = b_{ST}([\p{u},p], [\p{u}_h,p_h]).$$
	By adding up the above inequalities, the statement follows with $\mu_b \coloneqq \sum_i C_i$.
\end{proof}

Next we introduce NURBS projections, i.e. projections onto NURBS spaces, which can be used to measure the approximation properties of the test function spaces.

\begin{lemma}
	\label{NURBS_spprox1}
	Let $v \in H^s(\mathcal{Q}), \ s \geq \max\{s_1,s_2 \} $ with $ 1 \leq s_1 \leq r_T+1, \  1 \leq s_2 \leq r_S+1 $ and $\mathcal{V}_{h, r_S, r_T}$  the space-time NURBS space with an underlying regular mesh. \\
	Then there exists a projection  $\Pi_h \colon H^1(\mathcal{Q}) \rightarrow \mathcal{V}_{h,r_S,r_T}$ and  constants $C_{\Pi,j} $ not depending on $h_S, \, h_T$ and $v$ such that
	\begin{align*}
	\norm{\Pi_h v-v}_{L^2(\mathcal{Q})} &\leq C_{\Pi,1} (h_T^{s_1}+h_S^{s_2}) \norm{v}_{H^s(\mathcal{Q})}, \\
	\norm{\partial_{x_i} \big( \Pi_hv-v\big)}_{L^2(\mathcal{Q})} &\leq C_{\Pi,2}  (h_T^{s_1-1}+h_S^{s_2-1}) \norm{v}_{H^s(\mathcal{Q})}, \\
	\norm{\partial_t \big( \Pi_hv-v\big)}_{L^2(\mathcal{Q})} &\leq C_{\Pi,3}  (h_T^{s_1-1}+h_S^{s_2-1}) \norm{v}_{H^s(\mathcal{Q})}.
	\end{align*}
	If additionally $s \geq \max\{ s_1+1, s_2+1\}$, it holds moreover
	\begin{align*}
	\norm{\partial_{x_i}\partial_t \big( \Pi_hv-v\big)}_{L^2(\mathcal{Q})} &\leq C_{\Pi,4}  (h_T^{s_1-1}+h_S^{s_2-1}) \norm{v}_{H^s(\mathcal{Q})}, \\
	\norm{\partial_{x_i} \big( \Pi_hv-v\big)}_{L^2(\mathcal{Q})} &\leq C_{\Pi,5}  (h_T^{s_1}+h_S^{s_2-1}) \norm{v}_{H^s(\mathcal{Q})}.
	\end{align*} 
	The approximation results are also valid in case of homogeneous Dirichlet boundary conditions on the whole or on a part of the boundary.
\end{lemma}
\begin{proof}
	In the following, $C$ denotes a constant that depends only on the parametrization but may change at different occurrences. 
	The	main idea of the proof is the application of IGA approximation results presented in \cite{IGA1} (Part 3). For a better understanding we define the  auxiliary derivatives \begin{align}
	\label{aux_weak}
	D_{\tilde{\p{\Phi}}}^{\p{s}} u \coloneqq  \big( \partial_{\zeta_1}^{s_1} \dots \partial_{\zeta_d}^{s_d} \partial_{\zeta_{d+1}}^{s_{d+1}}(u \circ \tilde{\p{\Phi}}) \big) \circ \tilde{\p{\Phi}}^{-1}, \ \ \ \p{s} =(s_1,\dots,s_{d+1})\in \mathbb{N}^{d+1},
	\end{align} for sufficiently regular mappings $u$. Here $\partial_{\zeta_{i}}$ stands for the derivative w.r.t. the $i$-th coordinate in the parametric domain  $\hat{\mathcal{Q}}$. Latter definition is analogous to (56) in Part 3 of the mentioned reference.
	Let $v \in H^2(\mathcal{K})$ and $\mathcal{K}$ an element of the physical mesh  $\mathcal{M}$. The next step relates the weak derivatives to the definition \eqref{aux_weak}. 
	By the chain rule and the regularity of $\tilde{\p{\Phi}}$ we have 
	$$ \partial_t v= \partial_t \big[ \big(v \circ \tilde{\p{\Phi}}\big) \circ \tilde{\p{\Phi}}^{-1} \big]=  \hat{\nabla} \big(v \circ \tilde{\p{\Phi}} \big) \circ \tilde{\p{\Phi}}^{-1}  \cdot \partial_t \tilde{\p{\Phi}}^{-1}.$$
	We denote with $\hat{\nabla}$ the gradient w.r.t. to the coordinates $ \ \zeta_1, \dots,\zeta_{d+1} $ of the parametric domain $\hat{\mathcal{Q}}$. 
	Further rearrangements yield with the structure of $\tilde{\p{\Phi}}$ (see \eqref{eq:tensor_structure}),
	\begin{align}
	\label{dd1}
	\partial_t v =  \frac{1}{T} \,  \partial_{\zeta_{d+1}}\big(v \circ \tilde{\p{\Phi}}\big) \circ \tilde{\p{\Phi}}^{-1}= \frac{1}{T} D_{\tilde{\p{\Phi}}}^{(0,\dots,0,1)} v .
	\end{align}
	Moreover, using again the chain rule  it holds
	\begin{align*}
	\partial_{x_i} \partial_t v&= \frac{1}{T}\partial_{x_i} \big[ \partial_{\zeta_{d+1}}\big(v \circ \tilde{\p{\Phi}}\big) \circ \tilde{\p{\Phi}}^{-1} \big]=\frac{1}{T} \hat{\nabla} \big[ \partial_{\zeta_{d+1}} \big(v \circ \tilde{\p{\Phi}}\big) \big]  \circ \tilde{\p{\Phi}}^{-1}  \cdot \partial_{x_i} \tilde{\p{\Phi}}^{-1} \\
	&= \frac{1}{T}    \sum_{j=1}^d \big[ \partial_{\zeta_{j}} \partial_{\zeta_{d+1}}\big(v \circ \tilde{\p{\Phi}}\big) \circ \tilde{\p{\Phi}}^{-1} \big]  \, \big(\partial_{x_i} \tilde{\p{\Phi}}^{-1}\big)_j \ .
	\end{align*}
	Let $\p{e}_j \in \mathbb{N}^{d+1}$	 be the $j$-th canonical basis vector.
	Then we can choose  a constant $C=C(  \tilde{\p{\Phi}}^{-1})$  such that 
	\begin{align}
	\label{dd2}
	|\partial_{x_i} \partial_t v| \leq  C \, \sum_{j=1}^d \big| \partial_{\zeta_{j}} \partial_{\zeta_{d+1}}\big(v \circ \tilde{\p{\Phi}}\big) \circ \tilde{\p{\Phi}}^{-1}\big|=  C \, \sum_{j=1}^d \big| D_{\tilde{\p{\Phi}}}^{\p{e}_j+(0,\dots,0,1)} v  \big|.
	\end{align} 	
	Consequently, \eqref{dd1} and $\eqref{dd2}$ imply for some (new) constant $C$, only depending on the \\ parametrization, that
	\begin{align}
	\label{dd3}
	\norm{\partial_tv}_{L^2(\mathcal{K})} &\leq C \,  \norm{D_{\tilde{\p{\Phi}}}^{(0,\dots,0,1)} v }_{L^2(\mathcal{K})}, \\
	\norm{\partial_{x_i} \partial_t v}_{L^2(\mathcal{K})} &\leq C \,  \sum_{j=1}^d  \norm{D_{\tilde{\p{\Phi}}}^{\p{e}_j+(0,\dots,0,1)} v }_{L^2(\mathcal{K})} .
	\end{align}
	Similarly one gets 
	\begin{align}
	\norm{\partial_{x_i}v}_{L^2(\mathcal{K})} &\leq C \, \sum_{j=1}^d \norm{ D_{ \tilde{\p{\Phi}}}^{\p{e}_j} v }_{L^2(\mathcal{K})}, \\
	\norm{ v}_{L^2(\mathcal{K})} &=  \norm{  D_{\tilde{\p{\Phi}}}^{(0,\dots,0)} v }_{L^2(\mathcal{K})} .
	\label{dd4}
	\end{align}
	Again by the chain rule and an induction argument one gets a reverse estimate, namely
	\begin{align}
	\label{dd5}
	\norm{	D_{\tilde{\p{\Phi}}}^{\p{s}} u }_{L^2(\mathcal{K})} \leq C \norm{v}_{H^{s}(\mathcal{K})}, \ \ \textup{for} \ \  v \in H^s(\mathcal{K}) \ \ \textup{with} \   \ s=s_1+ \dots + s_{d+1}.
	\end{align}
	Observing the regularity of the mesh and that $h_T$ and $h_S$ denote the mesh sizes in the spatial domain and in the time interval, the assertion follows from the inequalities \eqref{dd3} - \eqref{dd4}  and \eqref{dd5} together with Theorem 7 in Part 3 of \cite{IGA1}. Therein the proof in detail is only shown for the two-dimensional case, i.e. in our setting the  case  $\mathcal{Q} \subset \mathbb{R}^{2}$.  But  the authors  remark the possibility to generalize the proofs and results straightforwardly also for higher-dimensional spaces.
	For the case of homogeneous boundary conditions one gets similar estimates due to Remark 14  in Part 3 of \cite{IGA1}. \\
	This finishes the proof.
\end{proof}

\begin{remark}
	\label{remark : projections}
	In view of  the last lemma we can incorporate homogeneous Dirichlet boundary conditions on the whole or part of the boundary without changing the approximation behavior of the NURBS spaces or NURBS projections, respectively. Thus we find projections
	\begin{align}
	\Pi_h^{\mathcal{W}} &\colon \mathcal{W}_0 \rightarrow \mathcal{W}_{h,r_p,r_T}, \\
	\Pi_h^{\p{\mathcal{V}}} &\colon \p{\mathcal{V}}_0 \rightarrow \mathcal{\p{\mathcal{V}}}_{h,r_u,r_T},  \textup{ with } \ \ \Pi_h^{\p{\mathcal{V}}}  \ \textup{acting componentwise, i.e.  } \big(\Pi_h^{\p{\mathcal{V}}}\p{u} \big)_i  \coloneqq \Pi_h^{\mathcal{V}} u_i, \ \textup{ where} 
	\end{align}
	$\Pi_h^{\mathcal{W}}, \ \Pi_h^{\mathcal{V}}$ satisfy the same  estimates as $\Pi_h$ in the last lemma,  potentially with new constants.
\end{remark}

Now we can state an approximation result for the  NURBS spaces in the norms  $\norm{\cdot}_h$ \textup{and} $\norm{\cdot}_{h,\star}$.
\begin{lemma}
	\label{IGA_space_time_estimates_Biot}
	Let  $ s_1,s_2,s_3 \in \mathbb{N}_{\geq 1} $, \ $ s_1 \leq r_T+1, \  s_2 \leq r_p+1, \ s_3 \leq  r_u+1$, where $r_p, \, r_u$ denote the polynomial degrees in the spatial coordinates and $r_T$ the polynomial degree in the temporal parameter. And let $l_1\geq\max\{s_1, \, s_2\}$ and $l_2 \geq \max \{s_1+1,s_3+1 \} $. \\
	If $p \in \mathcal{W}_0 \cap H^{l_1}(\mathcal{Q})$ and $\p{u} \in \p{\mathcal{V}}_0 \cap \p{H}^{l_2}(\mathcal{Q})$, then it holds
	\begin{align}
	\label{eq:iga_estimates1}
	\norm{\p{\Pi} [\p{u},p]-[\p{u},p]}_{h,\star} &\leq C_{\p{\Pi},1} \, \big(h_T^{s_1-1}+ h_S^{s_2-1}+ h_T^{-0.5}( h_S^{s_2}+h_S^{s_3-1}) \big) \ \mathcal{C}_{l_1,l_2}(\p{u},p), \\ 	\label{eq:iga_estimates2}
	\norm{\p{\Pi} [\p{u},p]-[\p{u},p]}_{h} &\leq C_{\p{\Pi},2} \,  (h_T^{s_1-1} + h_S^{s_2-1}+h_S^{s_3-1}) \ \mathcal{C}_{l_1,l_2}(\p{u},p) ,	\\
	\textup{where} \ \  \ \ \mathcal{C}_{l_1,l_2}(\p{u},p) &\coloneqq  \sqrt{\norm{\p{u}}_{\p{H}^{l_2}(\mathcal{Q})}^2+\norm{p}_{H^{l_1}(\mathcal{Q})}^2}. \nonumber
	\end{align}
	Here the constants $C_{\p{\Pi},1} ,C_{\p{\Pi},2} $ are independent of the mesh sizes $h_T, \ h_S$ and the mappings $p$ and $\p{u}$, but also independent of $c_0$. We use the notation $\p{\Pi} [\p{u},p]-[\p{u},p] \coloneqq [{\p{\Pi}_h^{\p{\mathcal{V}}}}\p{u}-\p{u},{\Pi}_h^{\mathcal{W}}p-p]$.
\end{lemma}
\begin{proof}
	We prove the statement for the norm $\norm{\cdot}_h$  first and consider each term in its definition \eqref{hnorm} separately. Doing so we use several times the estimates of Lemma \ref{NURBS_spprox1} and Remark \ref{remark : projections}. Besides we indicate by $C$ a constant which might change at different occurrences but is independent of the mesh sizes $h_T, \ h_S$ and the functions $\p{u}, \ p$.
	\begin{itemize}
		\item The first summand without the prefactor $h_T$ gives us with a  suitable constant $C$ and Lemma \ref{NURBS_spprox1}:
		\begin{align*}
		\norm{\partial_t \big( \p{\Pi}_h^{\p{\mathcal{V}}}\p{u}-\p{u}\big)}_{\p{\mathcal{H}}^1(\mathcal{Q})}^2 &\leq    \sum_{i,j}  \norm{\partial_{x_i} \partial_t(\p{\Pi}_h^{\mathcal{V}}u_j-u_j)}_{L^2(\mathcal{Q})}^2  \\ & \hspace{0.5cm} +\sum_{i} \norm{ \partial_t(\p{\Pi}_h^{\mathcal{V}}u_i-u_i)}_{L^2(\mathcal{Q})}^2 \\ &\leq C (h_T^{s_1-1}+h_S^{s_3-1})^2 \norm{\p{u}}_{\p{H}^{l_2}(\mathcal{Q})}^2 .
		\end{align*}
		\item For the   second  term we use integration by parts, the chain rule, and Lemma \ref{NURBS_spprox1}:
		\begin{align*}
		\norm{\p{\Pi}_h^{\p{\mathcal{V}}}\p{u}-\p{u}}_{\p{H}^1(\Sigma_T)}^2&= \sum_i \norm{\p{\Pi}_h^{{\mathcal{V}}}u_i-u_i}_{H^1(\Sigma_T)}^2 \\
		& =  \sum_{i,j} \int_{\mathcal{Q}} \partial_t \big[ \  \big(\partial_{x_j}(\p{\Pi}_h^{{\mathcal{V}}}u_i-u_i) \big)^2 \  \big] \, d\p{x}dt  \\ & \hspace{1cm}+ \sum_i \int_{\mathcal{Q}} \partial_t \big[ \ \big({\Pi}_h^{{\mathcal{V}}}u_i-u_i )^2 \ \big] \, d\p{x}dt \\
		& \leq C  \norm{\partial_t(\p{\Pi}_h^{\p{\mathcal{V}}}\p{u}-\p{u})}_{\p{\mathcal{H}}^1(\mathcal{Q})} \,  \norm{\p{\Pi}_h^{\p{\mathcal{V}}}\p{u}-\p{u}}_{\p{\mathcal{H}}^1(\mathcal{Q})} \\
		& \leq C   (h_T^{s_1-1}+h_S^{s_3-1})^2 \norm{\p{u}}_{\p{H}^{l_2}(\mathcal{Q})}^2.
		\end{align*}
		In the second line, we assumed $\partial_{x_j}u=0$ on $\Sigma_0$ in the sense of the trace theorem. This is indeed correct due to the  $\p{H}^2$-regularity of $\p{u}$ and the condition $\p{u}=0$ on $\Sigma_0$.
		\item In an analogous fashion,  one has
		\begin{align*}
		\norm{\Pi_h^{\mathcal{W}}p-p}_{L^2(\Sigma_T)}^2 \leq   C   \,  (h_T^{s_1-1}+h_S^{s_2-1})^2 \norm{p}_{H^{l_1}(\mathcal{Q})}^2 ,
		\end{align*}
		where we used again Lemma \ref{NURBS_spprox1}.
		\item Furthermore we obtain with the above lemma
		\begin{align*}
		 \norm{\partial_t \big(\Pi_h^{\mathcal{W}}p-p\big)}_{L^2(\mathcal{Q})}^2+\norm{\nabla_x \big( \Pi_h^{\mathcal{W}}p-p)}_{\p{L}^2(\mathcal{Q})}^2 
	\\	\hspace{1cm} \leq C  (h_T^{s_1-1}+h_S^{s_2-1})^2 \norm{p}_{H^{l_1}(\mathcal{Q})}^2.
		\end{align*}
	\end{itemize}
	The last four estimates yield 
	\begin{align}
	\norm{\p{\Pi} [\p{u},p]-[\p{u},p]}_h^2  \nonumber  
	&\leq C \, \Big[ \big( \ h_T\,(h_T^{s_1-1}+h_S^{s_3-1})^2  +   (h_T^{s_1-1}+h_S^{s_3-1})^2 \  \big) \norm{\p{u}}_{\p{H}^{l_2}(\mathcal{Q})}^2  \\ & \hspace{1.1cm} +    \,  \big( \  (h_T^{s_1-1}+  h_S^{s_2-1})^2  +  (h_T^{s_1-1}+h_S^{s_2-1})^2  \ \big) \norm{p}_{H^{l_1}(\mathcal{Q})}^2 \Big] \nonumber \\
	\label{hnormesti1}
	& \leq C \, (h_T^{s_1-1}+h_S^{s_2-1}+h_S^{s_3-1})^2 ( \norm{\p{u}}_{\p{H}^{l_2}(\mathcal{Q})}^2+\norm{p}_{H^{l_1}(\mathcal{Q})}^2).
	\end{align}
	This implies \eqref{eq:iga_estimates2}. One notes the assumption $c_0 \leq 1$.\\
	For the $\norm{\cdot}_{h,\star}$-norm it is sufficient to consider the parenthesized term  in \eqref{def:h_star_norm} which gives with prefactor, using again Lemma \ref{NURBS_spprox1} and Remark \ref{remark : projections},
	\begin{align*}
	\frac{1}{h_T} &\norm{\p{\Pi}_h^{\p{\mathcal{V}}}\p{u}-\p{u}}_{\p{\mathcal{H}}^1(\mathcal{Q})}^2 +  \frac{1}{h_T}  \norm{\Pi_h^{\mathcal{W}}p-p}_{L^2(\mathcal{Q})}^2\\
	&\leq C \frac{1}{h_T} (h_T^{s_1}+h_S^{s_3-1})^2  \norm{\p{u}}_{\p{H}^{l_2}(\mathcal{Q})}^2+C \, \frac{1}{h_T} (h_T^{s_1}+h_S^{s_2})^2 \norm{p}_{H^{l_1}(\mathcal{Q})}^2
	\\
	&\leq  C  \frac{1}{h_T}(h_T^{s_1}+h_S^{s_2}+h_S^{s_3-1})^2 ( \norm{\p{u}}_{\p{H}^{l_2}(\mathcal{Q})}^2+\norm{p}_{H^{l_1}(\mathcal{Q})}^2).
	\end{align*}
	The last three lines and the already shown  inequalities   give us the desired bound \eqref{eq:iga_estimates1} for the norm $\norm{\cdot}_{h,\star}$.
\end{proof}

Finally we are ready to prove a convergence estimate for the space-time method.

\begin{theorem}{\textup{\textbf{Convergence for smooth solutions}}}\\ 
	Let the assumptions of Lemma \ref{IGA_space_time_estimates_Biot} be fulfilled, where $l_2\geq 3, \ l_1 \geq 2$.\label{th:conv1}
	Moreover let $\p{u} \in \p{\mathcal{V}}_{0} \cap \p{H}^{l_2}(\mathcal{Q})$, \ $p\in \mathcal{W}_{0} \cap H^{l_1}(\mathcal{Q})$ be the exact solution of the Biot system \eqref{Eq1}-\eqref{Eq2} with initial-boundary conditions \eqref{init_bound_con_1}-\eqref{init_bound_con_4} . \\
	Then we have for the error  between $\p{u},p$ and the solution $\p{u}_h,p_h$ of the finite-dimensional variational problem \eqref{disctre_variational_problem_short}
	\begin{align}
	\label{convergence_equation}
	\norm{[\p{u},p]-[\p{u}_h,p_h]}_h \leq C \big( h_T^{s_1-1}+h_S^{s_2-1}+h_T^{-0.5}  (h_S^{s_2} +  h_S^{s_3-1} )  \big)  \ \mathcal{C}_{l_1,l_2}(\p{u},p),
	\end{align}
	where $C$ is some constant, which does not depend on $c_0$ and the mesh sizes $h_T, \ h_S$.
\end{theorem}
\begin{proof}
	We make use of the coercivity and continuity of $b_{ST}$, shown in Lemma \ref{lemma:coercivity} and Lemma \ref{lemma:continuityA} and obtain with the NURBS projections  $\p{\Pi} [\p{u},p] \coloneqq [\p{\Pi}_h^{\p{\mathcal{V}}}\p{u},\Pi_h^{\mathcal{W}}p]$, $\Pi_h^{\mathcal{W}} \colon \mathcal{W}_0  \rightarrow \mathcal{W}_{h,r_p,r_T}$ and  $\p{\Pi}_h^{\p{\mathcal{V}}} \colon \p{\mathcal{V}}_0  \rightarrow \p{\mathcal{V}}_{h,r_u,r_T} $: 
	\begin{align*}
	\mu_c \, \norm{\p{\Pi} [\p{u},p]-[\p{u}_h,p_h]}_{h}^2 &\leq b_{ST}\big(\p{\Pi} [\p{u},p]-[\p{u}_h,p_h] \ ,\ \p{\Pi} [\p{u},p]-[\p{u}_h,p_h]\big) \\
	& = b_{ST}\big(\p{\Pi} [\p{u},p]-[\p{u},p] \ , \ \p{\Pi} [\p{u},p]-[\p{u}_h,p_h]\big) \\
	& \leq \mu_b \, \norm{\p{\Pi} [\p{u},p]-[\p{u},p]}_{h,\star}   \norm{\p{\Pi} [\p{u},p]-[\p{u}_h,p_h]}_{h}.
	\end{align*}  
	Note that we used the consistency result of Lemma \ref{lemma:consistency}.
	Above inequality chain implies
	\begin{align*}
	\norm{\p{\Pi} [\p{u},p]-[\p{u}_h,p_h]}_{h} \leq \frac{\mu_b}{\mu_c} \, \norm{\p{\Pi} [\p{u},p]-[\p{u},p]}_{h,\star}.
	\end{align*}
	We remark for the above inequality that $\p{\Pi} [\p{u},p]-[\p{u}_h,p_h] \in \p{\mathcal{V}}_{h,r_u,r_T} \times \mathcal{W}_{h,r_p,r_T}$. \\
	By the previous  inequality and Lemma \ref{IGA_space_time_estimates_Biot} it follows 
	\begin{align*}
	\norm{[\p{u},p]-[\p{u}_h,p_h]}_h &\leq \norm{\p{\Pi}[\p{u},p]-[\p{u}_h,p_h]}_h+\norm{\p{\Pi}[\p{u},p]-[\p{u},p]}_h 
	\\ & \leq \frac{\mu_b}{\mu_c} \, \norm{\p{\Pi} [\p{u},p]-[\p{u},p]}_{h,\star} + \norm{\p{\Pi}[\p{u},p]-[\p{u},p]}_h  \\
	& \leq  \frac{\mu_b}{\mu_c} \,   C_{\p{\Pi},1}  \big(h_T^{s_1-1}+ h_S^{s_2-1}+ h_T^{-0.5}( h_S^{s_2}+h_S^{s_3-1}) \big) \ \mathcal{C}_{l_1,l_2}(\p{u},p) \\ & \hspace{0.5cm} +  C_{\p{\Pi},2} \, (h_T^{s_1-1} + h_S^{s_2-1}+h_S^{s_3-1})  \ \mathcal{C}_{l_1,l_2}(\p{u},p) \\
	& \leq C \,  \big(h_T^{s_1-1}+ h_S^{s_2-1}+ h_T^{-0.5}( h_S^{s_2}+h_S^{s_3-1}) \big) \ \mathcal{C}_{l_1,l_2}(\p{u},p), 		
	\end{align*}
	for some constant $C$, e.g. $C=\big(\frac{\mu_b}{\mu_c} C_{\p{\Pi},1}+ C_{\p{\Pi},2} \big)$ if $h_T \leq 1$. 
\end{proof}

How does this convergence result compare with the vertical method of lines? In the latter case, the 
spatial discretization leads to an ODE or DAE in time which is then solved by specific integration schemes.
In Theorem \ref{th:conv1}, the spatial and temporal influences show up in   terms of the type $h_S^{a}$,  $h_T^{a}$ respectively. If we assume $h_S$ of order $\mathcal{O}(h_T)$ and require a regular exact solution then the error w.r.t. the norm $\norm{\cdot}_{h}$ is of order $\mathcal{O}(h_T^{r_T}+h_S^{r_p}+h_T^{-0.5}(h_S^{r_p+1}+h_S^{r_u}))= \mathcal{O}(h_T^{r_T}+h_S^{r_p}+h_S^{r_u-0.5})$. 
Error estimates for the method of lines typically have the form $ \textup{error} \leq \textup{const.} \cdot  \big( \Delta t^{m} + h_S^k \big)$, where $m$ denotes the convergence order of the time integrator, $\Delta t$ the step size and $k$ the approximation order of the spatial discretization.

From a theoretical point of view it is possible to obtain high convergence orders in both ways. But for the space-time method increasing the convergence order can be achieved by raising polynomial degrees, which becomes very efficient in the context of IGA. 
Variable order multistep formulae also allow to do this, in particular the BDF, while
for implicit and linear-implicit one-step methods such as Runge-Kutta,  the order is fixed. Besides, for multistep methods  the choice of good starting values is of relevance and order reduction problems might occur for Runge-Kutta type algorithms; compare Remark 6 and 7 in \cite{FU}.

The case of the continuous P1-spline discretization in time in the space-time method deserves a further remark. For illustration, we apply the method based on the time-upwind test functions introduced in \cite{para}, which is similar to our discretization
of the Biot system,  to the semi-discretized homogeneous heat equation $\p{M}_h \partial_t{\p{w}} = \p{A}_h\p{w}$. 
Here we have an  spd mass matrix $\p{M}_h$ and a semi-definite right-hand side matrix $\p{A}_h$. More precisely, we make the ansatz $\p{w}=\sum_k  \p{w}_k \ \phi_k(t) $, where $\phi_k$ are P1-splines, and multiply the system by test functions $ \big( \phi_k(t)+\Delta t \partial_t \phi_k(t) \big) \p{e}_i  $, where $\p{e}_i$ denotes the $i$-th canonical basis vector, and integrate w.r.t. time. Instead of a simultaneous space-time discretization, we split the procedure thus into two discretization steps and obtain, after some straightforward computations, the recursion formula
\begin{equation}\label{eq:mstanalogy}
\p{M}_h \left( -\frac{3}{4}\p{w}_{k-1} + \p{w}_k - \frac{1}{4}\p{w}_{k+1} \right)
=   \Delta t \p{A}_h \left(\frac{1}{3} \p{w}_{k-1}  + \frac{1}{3} \p{w}_{k} -  \frac{1}{6} \p{w}_{k+1}\right)
\end{equation}
for the time steps $k=1,2,\ldots$. We observe the structure of an implicit 2-step method, and the consistency order is readily checked to be $m=3$. The difference to a classical multistep approach that would proceed step by step lies in the first step $k=1$ where both $\p{w}_1$ and $\p{w}_2$ are unknowns as there is no starting procedure for $\p{w}_1$. 
If we arrange all time steps in a large linear system, it will  have a tridiagonal staircase block structure that reflects the propagation of information.
Thus, the P1-spline temporal discretization is inappropriate for a simple sequential processing in time. The same reasoning applies to the full space-time method for the Biot system. It is thus natural to treat the  fully discretized system en bloque
by suitable sparse direct or iterative methods that in the end take implicitly advantage of the staircase structure.

Finally, the   factor $h_T^{-0.5}$ in the estimate \eqref{convergence_equation} suggests an order reduction effect. 
Our numerical experiments below, however, did not reveal this, which indicates that our estimate might be too pessimistic.


\section{Numerical examples}
\label{sec:numerical_examples}
In this section we focus on numerical examples for validating the convergence behaviour, but also instability issues are considered.
The numerical experiments are performed by means of the powerful
GeoPDEs package \cite{geopdes,geopdes3.0}, which is a 
MATLAB \cite{MATLAB:2019} implementation of IGA.
We write IG-ST as an abbreviation for the space-time method introduced above. Further, the following test examples also have been computed by means of a method-of-lines ansatz, namely a spatial isogeometric discretization combined with a BDF time-stepping. The obtained results are not documented here, but were used to check the plausibility of the IG-ST solutions.
The overall linear system of the space-time method is solved by MATLAB's sparse direct solver. This means that we concentrate
here on checking convergence and treat the linear algebra as a black box. Of course, there is room for substantial adaptation and tuning in this regard.  

\subsection{Numerical convergence analysis} 
\label{subsec:numerics_convergence}
By manufacturing the right-hand sides, we construct the smooth strong solution 
\begin{align}
u_1(\p{x},t)&=\sin( \pi \, x ) \sin(\pi \, y) \, \sin( \pi \, t),  \qquad
u_2(\p{x},t)=\sin( \pi \, x ) \sin(\pi \, y) \, (\exp(t)-1),  \label{eq:exact_sol1} \\
p(\p{x},t)&=\sin( \pi \, x ) \sin(\pi \, y) \,\sin(0.5 \, \pi \, t) ,  \ \ \   (x,y)\coloneqq \p{x} , \nonumber
\end{align}
of the Biot system where the spatial domain is  the unit square and the  time interval is $[0,T]=[0,1]$. Hence the space-time 
cylinder $\mathcal{Q}$ is the three-dimensional unit cube. 
Uniform meshes $\mathcal{M}_h$ are obtained by dividing the cube into equally smaller cubes with edge lengths $h=2^{-k}, \ k=1, \dots,5$.  In other words we ignore here the difference between spatial and temporal mesh sizes and have  $h=h_T=h_S$. We assume homogeneous initial-boundary conditions on the surface of the space-time cylinder. 
For the coefficients and parameters, respectively, we set  $c_0=1 , \ \lambda=1, \ \mu=1, \ \p{\mathcal{K}}= \p{I}, \ b=1 \hspace{0.2cm}.$ 
For the underlying discrete NURBS spaces we use simple inner knots, and in order to save  computational costs we apply basis functions which are also $r_T-1$ times continuously differentiable in time.

First we choose for the polynomial degrees a mixed ansatz in space, namely $r_p=r_T$ and $r_u=r_p+1$. Due to  regularity of the strong solution and the choice of the boundary conditions, we expect the error in the norm $\norm{\cdot}_h$ to be of order $\mathcal{O}(h^{r_p}+h^{r_u-0.5})$; see  Theorem \ref{th:conv1}. We compute for the different meshes $\mathcal{M}_h , \ h=1/2, \dots, 1/32$ and the polynomial degrees $r_p=1,2,3$ the errors w.r.t. the $\norm{\cdot}_h$ norm. The result is shown in the Fig. \ref{fig:conv} (a) above.
\begin{figure}[h]
	\begin{center}		
		\begin{minipage}{5.9cm}
			\begin{center}
				\includegraphics[scale=0.46]{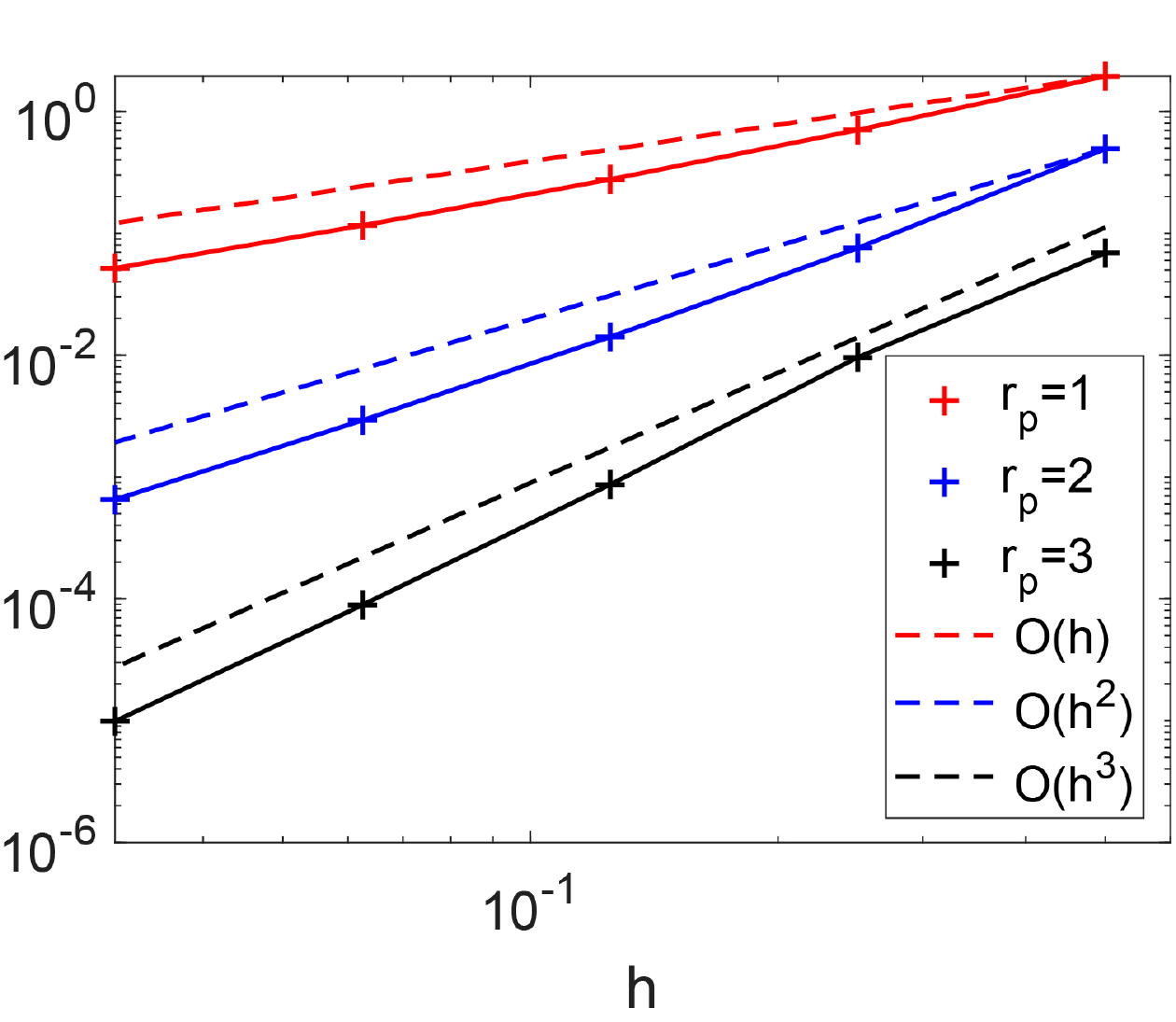}
				{ \small (a) \ $r_p=r_T, \ r_u=r_p+1$ and $c_0=1$. }
			\end{center}		
		\end{minipage}
		\hspace{0.5cm}
		\begin{minipage}{6cm}
			\begin{center}
				\includegraphics[scale=0.46]{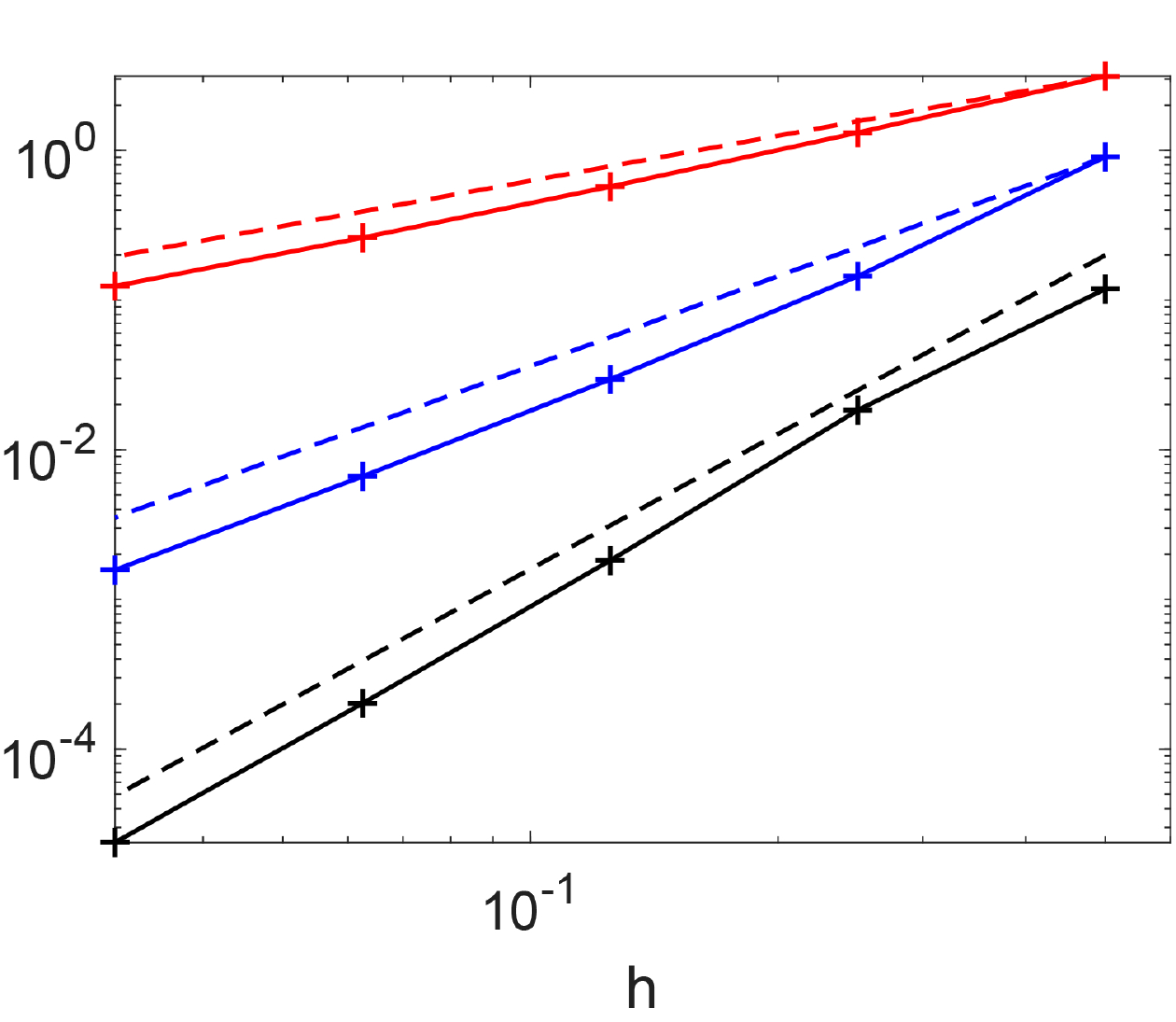}
			{ \small (b)  \ $r_p=r_T=r_u$ and $c_0=1$. }
			\end{center}		
		\end{minipage}		
		\begin{minipage}{5.9cm}
			\begin{center}
				\includegraphics[scale=0.46]{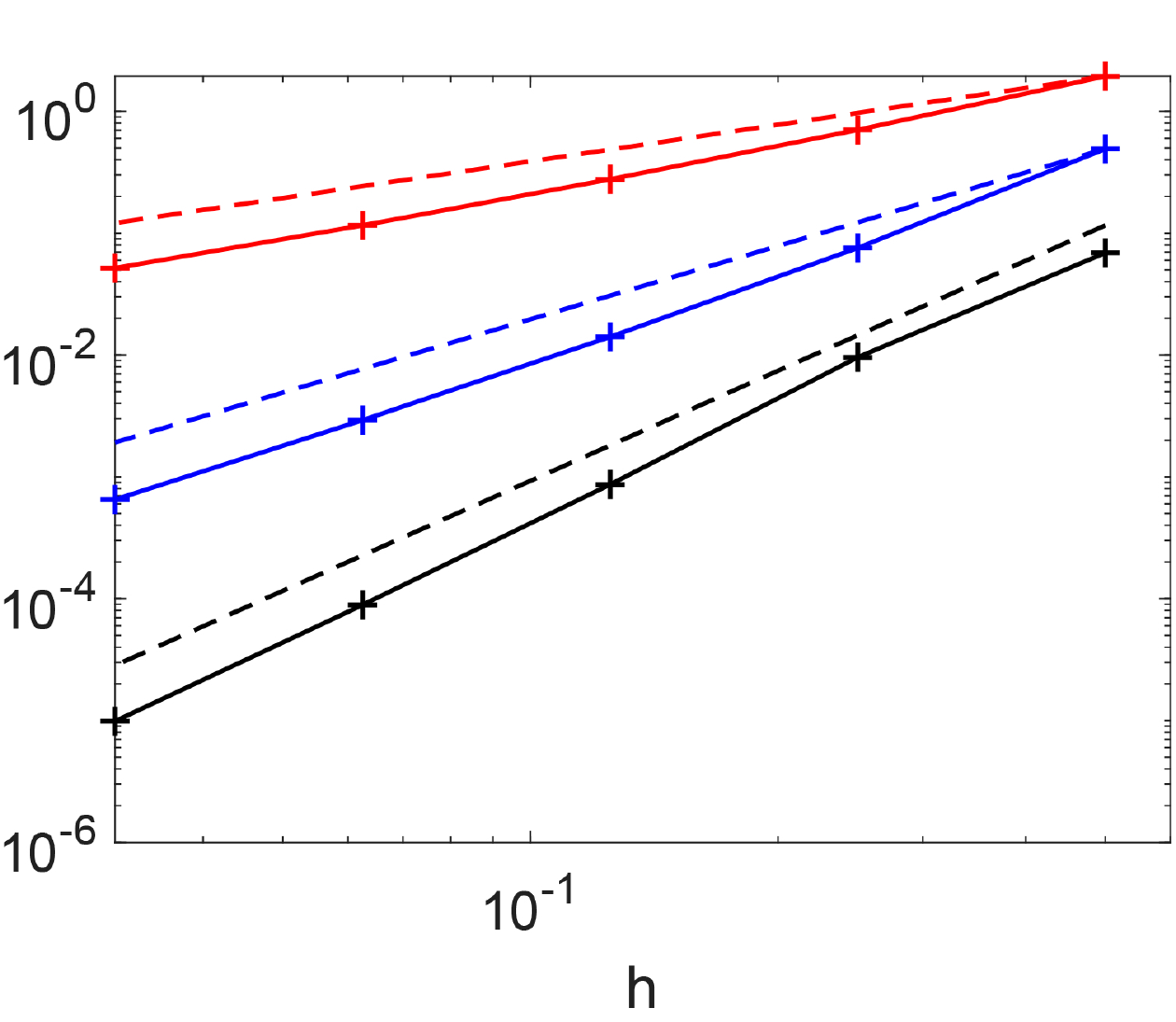}
				{ \small (c) \  $r_p=r_T, \ r_u=r_p+1$ and $c_0=0$. }
			\end{center}		
		\end{minipage}
		\hspace{0.5cm}
		\begin{minipage}{6cm}
			\begin{center}
				\includegraphics[scale=0.46]{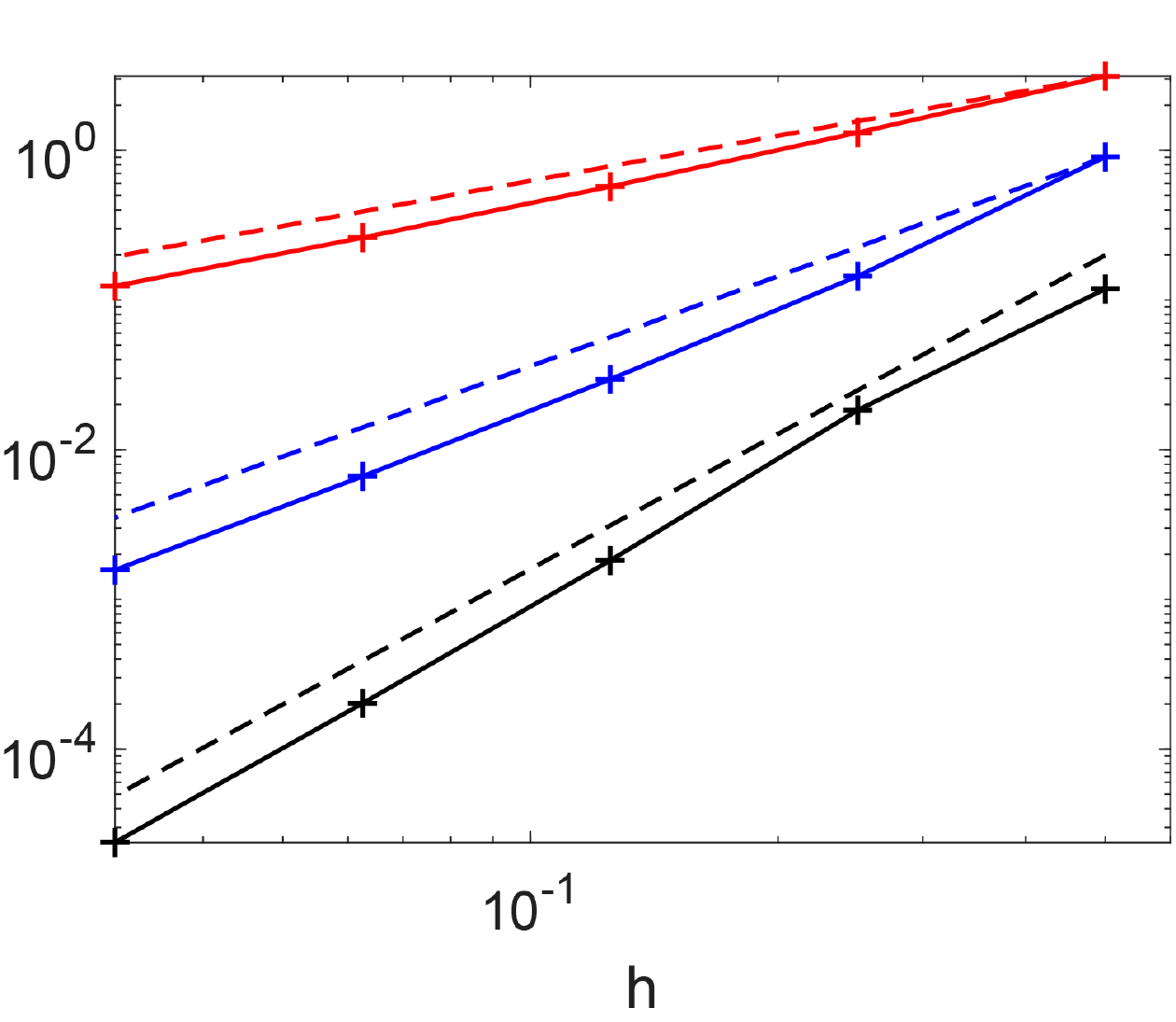}
				{ \small (d) \  $r_p=r_T=r_u$ and $c_0=0$. }	
			\end{center}		
		\end{minipage}	
		\caption{\small Errors in the norm $\norm{\cdot}_h$ for mixed and non-mixed polynomial degrees as well as for the two cases $c_0=1$ and $c_0=0$.}
		\label{fig:conv}
	\end{center}
\end{figure}

Additionally  we display in the Fig. \ref{fig:conv} (b) the computed $\norm{\cdot}_h$-norm  errors, but  with  equal degrees  $r_p=r_T=r_u$ for $r_p=1,2,3$. One observes a similar convergence behaviour like for the mixed-degrees case although our convergence estimate indicates an order reduction. Thus one might conjecture that the convergence estimates are not yet optimal.
To demonstrate the convergence behaviour for the extreme case $c_0=0$, we display in the Figures  \ref{fig:conv} (c)-(d)  the $\norm{\cdot}_h$-norm errors for the above test problem with $c_0$ set to zero.

\subsection{Terzaghi's problem and problem of Barry and Mercer}

\label{subsubsection:code_eval}
Next we give two examples from the literature for which an analytical solution is known.
Terzaghi's problem is a one-dimensional model with analytical pressure solution that describes the coupling of the fluid pressure and deformation   of   a porous medium pipe completely filled with some fluid, if one end is fixed and at the other end a uniform normal  surface load $\p{F}=(F,0)$  is applied; see Fig. \ref{fig:Terzaghi_pipe}.  We require the displacement and fluid flow to be  restricted parallel to the pipe such that the problem can be reduced to a one-dimensional setting.
The governing equations  are given by the 1D Biot system where the physical domain $\Omega$ reduces to some interval $\Omega=(0,L)$ and the permeability tensor is given by  some positive constant $k$.
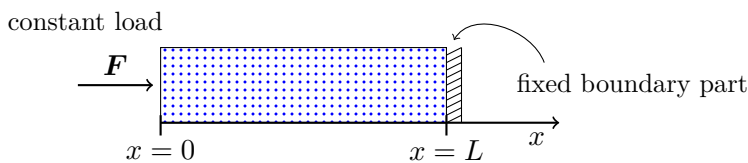
\begin{figure}[h]
	\begin{center}
		\begin{tikzpicture}
		\path [pattern=dots, pattern color=blue] (0.1,-0.5) rectangle (3.9,0.5);
		\draw (0.1,-0.5) rectangle (3.9,0.5);
		\draw[thick, ->] (0.1,-0.5) to (5.4,-0.5);
		\draw[thick] (0.1,-0.67) to (0.1,-0.4) ;
		\draw[thick] (3.9,-0.67) to (3.9,-0.4) ;
		\draw(4.1,-0.5)to (4.1,0.5);
		\draw (3.9,-0.5) to (4.1,-0.4);
		\draw (3.9,-0.4) to (4.1,-0.3);
		\draw (3.9,-0.3) to (4.1,-0.2);
		\draw (3.9,-0.2) to (4.1,-0.1);
		\draw (3.9,-0.1) to (4.1,-0);
		\draw (3.9,-0) to (4.1,0.1);
		\draw (3.9,0.1) to (4.1,0.2);
		\draw (3.9,0.2) to (4.1,0.3);
		\draw (3.9,0.3) to (4.1,0.4);
		\draw (3.9,0.4) to (4.1,0.5);
		\node[below] at (5.1,-0.5) {$x$};
		\draw[thick , ->] (-1,0) to (0,0);
		\node[above] at (-0.5,0) {$\p{F}$};
		\node[below] at (0.1,-0.6) {$x=0$};
		\node[below] at (3.9,-0.6) {$x=L$};
		\draw[<-, out=60, in=110] (4,0.6) to (5.2,0.3);
		\node[right] at (4.7,0) {\small fixed boundary part};
		\node[above] at (-0.9,0.6) {\small constant load};
		\end{tikzpicture}
	\end{center}
	\caption{ \small Terzaghi's problem: Saturated porous medium deformed by some loading.}
	\label{fig:Terzaghi_pipe}
\end{figure} 
The boundary and initial conditions are set to
\begin{align*}
p=0, \ \  \p{t}_n=F, \ 0<F<\infty   \hspace{1cm} & \textup{at } \ x=0, \\ 
\partial_x p=0, \ \ u=0 \hspace{1cm} & \textup{at } \ x=L, \ 0<L<\infty, \\
p(t=0)=u(t=0)=0. \hspace{0.9cm} &
\end{align*}
The exact pressure solution can be found in \cite{phillips}.  
We set the parameters to $c_0=0.2 , \ \lambda=1, \ \mu=1, \ k=0.2, \ b=1, \ \eta_f=1  \ \textup{and} \ F=L=1$,
and the space-time cylinder is  $\mathcal{Q}=(0,1) \times (0,2)$.  We use a uniform mesh with mesh sizes $h_T=0.01, \ h_S=0.05$ and polynomial degree $1$ w.r.t. each coordinate. In Fig. \ref{fig : Terzaghi1} (a) the numerical solution is displayed. The approximate values  fit to the exact solution quite well. One observes that the deviation is at maximum for $t=0.05$, which can be explained by the following
reasoning: The exact solution converges for $t \rightarrow 0$ pointwise to some discontinuous function $x \mapsto p_0 \, \chi_{(0,1]}(x), \ (\chi \ \textup{indicator function})$. But the IG-ST method is used with zero initial conditions and yields only globally continuous solutions. Thus the non-smooth limit behavior can not be reproduced by the IG-ST method and
the deviation between exact and numerical solution grows for $t\rightarrow 0$, as exemplified in Fig. \ref{fig : Terzaghi1} (b) for $t=0.01$. 
\begin{figure}[h]
	\begin{minipage}{7.08cm}
		\begin{center}
			\includegraphics[scale=0.5]{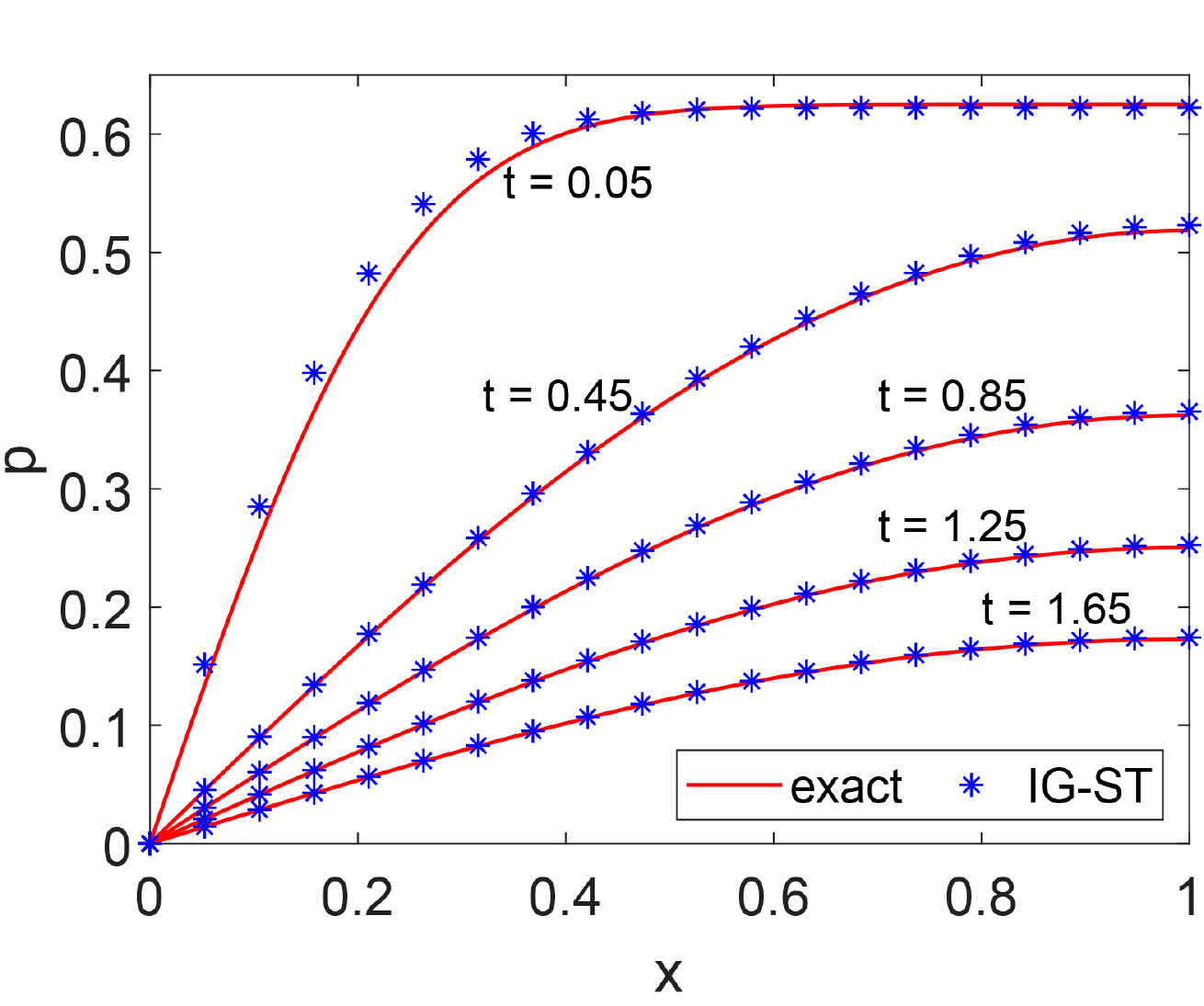}
			{ \small \ \ \  (a) Results for $h_T = 0.01, \ h_S = 0.05$ . }
		\end{center}
	\end{minipage}
	\hspace{0.9cm}
	\begin{minipage}{7.1cm} 			                    
       \begin{center}
       	 \includegraphics[scale=0.5]{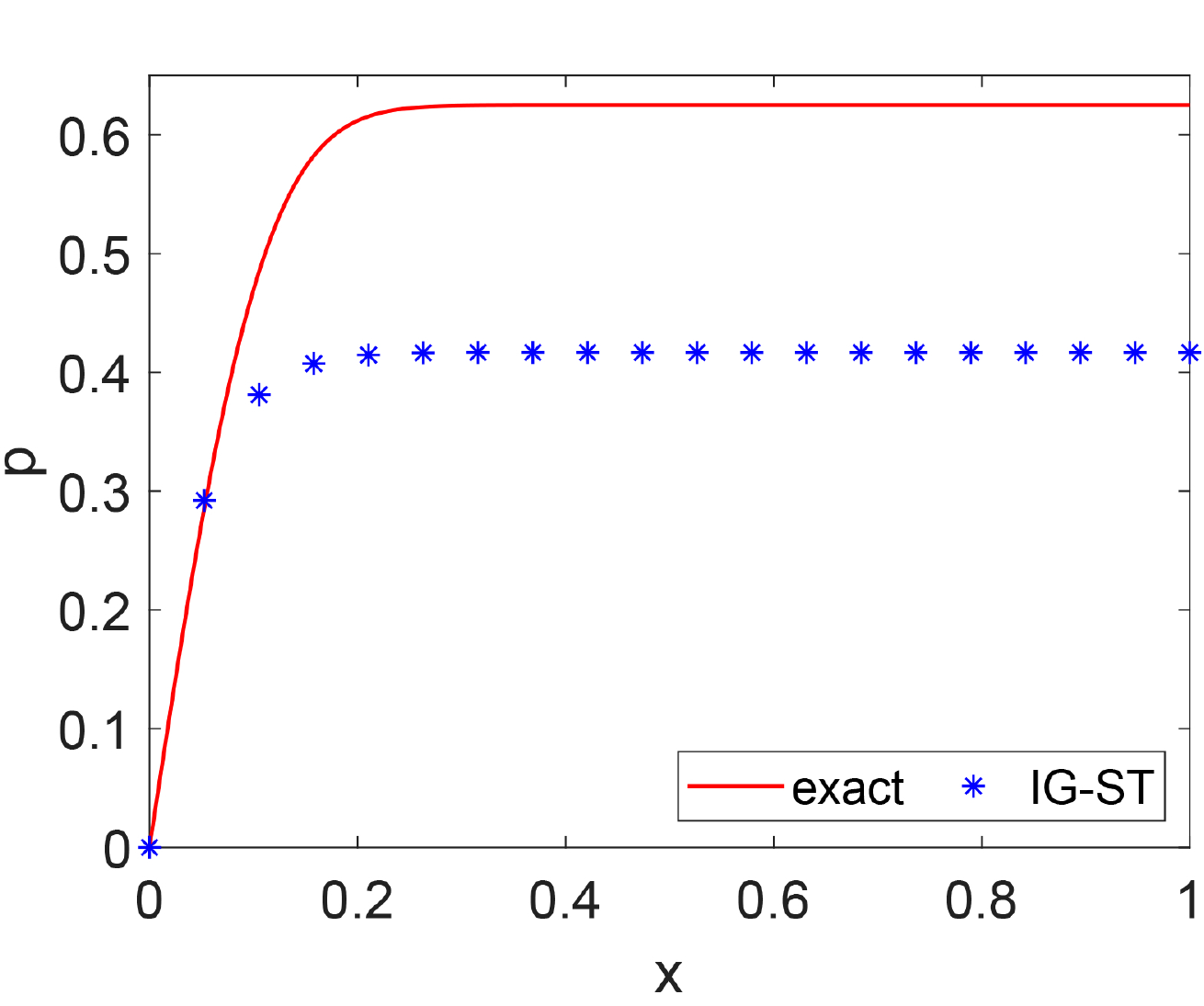}
       	 { \small (b) Early time solution at $t=0.01$ .}
       \end{center}
	\end{minipage}
	\caption{ \small Comparison between exact solution of the Terzaghi test problem and the approximate solutions. The deviation grows for $t \rightarrow 0$.}
	\label{fig : Terzaghi1}
\end{figure}
This drawback can be alleviated by using finer meshes. 

We proceed with  a two-dimensional test problem, the problem of Barry and Mercer, which is taken from \cite{FU}. This problem  describes the pressure and displacement in a rectangular porous medium under the influence of an oscillating fluid point source. Though it has   an analytical solution, it is  only available in the form of   infinite double series. And the source term is actually given by a distribution and not by a function. This implies the necessity to approximate the source term by a proper function.  

More precisely we have the setting $\Omega=(0,1)^2$ and the only source distribution is $g( \p{x},t)= 2 \, \beta \, \delta(\p{x}-\p{x}_0) \,  \sin(\beta t)$, where $\delta(\p{x}-\p{x}_0)$ denotes the Dirac delta distribution at $\p{x_0} \coloneqq (x_{0},y_{0})$ and $\beta= (\lambda+2\mu) k $. Further we use homogeneous Dirichlet boundary conditions for the pressure variable and a mixture of homogeneous Neumann and Dirichlet  boundary conditions for the displacement  variables;  see Fig. \ref{fig:Boundary_Condition_Barry_and_Mercer}. The analytical solution of the problem is stated in   Section 4.2.1 in \cite{phillips}. 
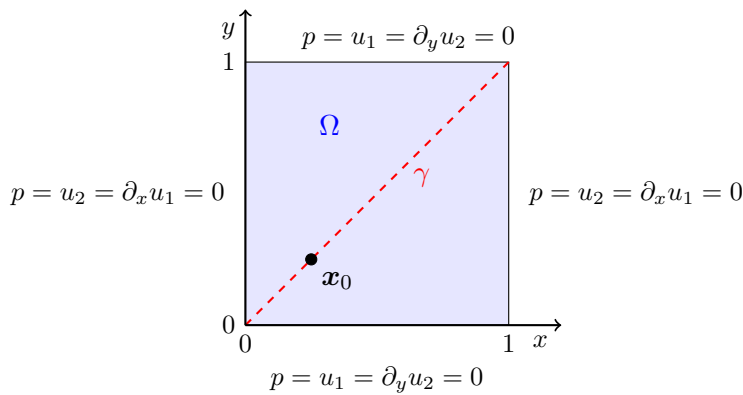
\begin{figure}[h]
	\begin{center}
		\begin{tikzpicture}[scale=1.4]
		\draw[fill=blue,opacity=0.1] (0,0) rectangle (2.5,2.5);
		\draw (0,0) rectangle (2.5,2.5);
		\draw[dashed,thick,red] (0,0) -- (2.5,2.5);
		\draw[fill] (0.625,0.625) circle (1.5pt);
		\node[below] at (1.25,-0.3) { \small $p=u_1= \partial_{y}u_2=0$}; 
		\node[above] at (1.55,2.5) { \small $p=u_1= \partial_{y}u_2=0$}; 
		\node[left] at (-0.1,1.25) { \small $p=u_2= \partial_{x}u_1=0$}; 
		\node[right] at (2.6,1.25) { \small $p=u_2= \partial_{x}u_1=0$}; 
		\node[right] at (0.625, 0.425) {$\p{x}_0$};
		\node[thick,blue] at (0.8,1.9) {$\Omega$};
		\draw[thick,->] (0,0) to (3,0);
		\draw[thick,->] (0,0) to (0,3);
		\node[right,red] at (1.5,1.4) {$\gamma$};
		\node[below] at (0,0) { \small $0$};
		\node[left] at (0,2.5) {\small $1$};
		\node[left] at (0,2.8) {\small $y$};
		\node[below] at (2.5,0) {\small $1$};
		\node[below] at (2.8,0) {\small $x$};
		\node[left] at (0,0) { \small $0$};
		\end{tikzpicture}
	\end{center}
	\caption{ \small Boundary conditions for the problem of Barry and Mercer.}
	\label{fig:Boundary_Condition_Barry_and_Mercer}
\end{figure}
The corresponding parameters, taken from  \cite{FU}, are
$c_0=0 , \ \lambda=10^4/0.88, \ \mu=10^5/2.2 , \ \p{\mathcal{K}}=k \cdot \p{I}, \ k= 0.01, \ b=1 \ \textup{and} \  \p{x}_0=(0.25,0.25) .$
Note that the approximation of the Dirac delta is realized in the following way.
We partition  $\Omega$ into equal squares and set  the edge length  $h$ of the squares  always such that $\p{x}_0$ is the center of one square $S_h$. Using this we approximate $\delta(\p{x}-\p{x}_0)$ by 
$\delta_h(\p{x}-\p{x}_0) =h^{-2} \ \textup{if } \ \p{x} \in S_h $ and $\delta_h(\p{x}-\p{x}_0)=0$ else.
As spatial mesh sizes we consider $h_S=1/34$ and $h_S=1/66$,  \rm{i.e.}, a coarser  and a finer mesh. The underlying polynomial degrees are one. Then we choose as space-time cylinder  $\mathcal{Q}=(0,1)^2 \times (0,\frac{3 \pi}{2 \, \beta})$ with mesh size in time  $h_T=\frac{3 \pi}{36 \, \beta}$. One notes that the spatial mesh size is much larger than the temporal mesh size. Hence the distinction between spatial and temporal step size for the space-time variational formulation seems to be reasonable.  

We compare the numerical solution with a reference solution which is  actually very close to the analytical one and  hence suitable for comparisons and we refer to it as exact solution.  For reasons of comparability we plot the exact and numerical solutions  of the pressure $p$ and displacement $u_1$ along the diagonal line $(0,0)-(1,1)$ (see $\gamma$ in Fig. \ref{fig:Boundary_Condition_Barry_and_Mercer}) of the domain at the times $t_1 = 0.5 \pi /\beta$ and $t_2 = 1.5 \pi /\beta$ in Fig. \ref{fig:Barry_Mercer_4}.
The approximate and reference solutions of the displacement in $x$ direction match quite well. The pressure solution deviates near the point source $\p{x}_0$ due to the coarseness of the mesh and the related approximation of the Dirac delta. For a finer spatial mesh the results are clearly better near $\p{x}_0$. 

\begin{figure}[h]
	\begin{center}
		\begin{minipage}{6.1cm}
	     \begin{center}
		\includegraphics[scale=0.46]{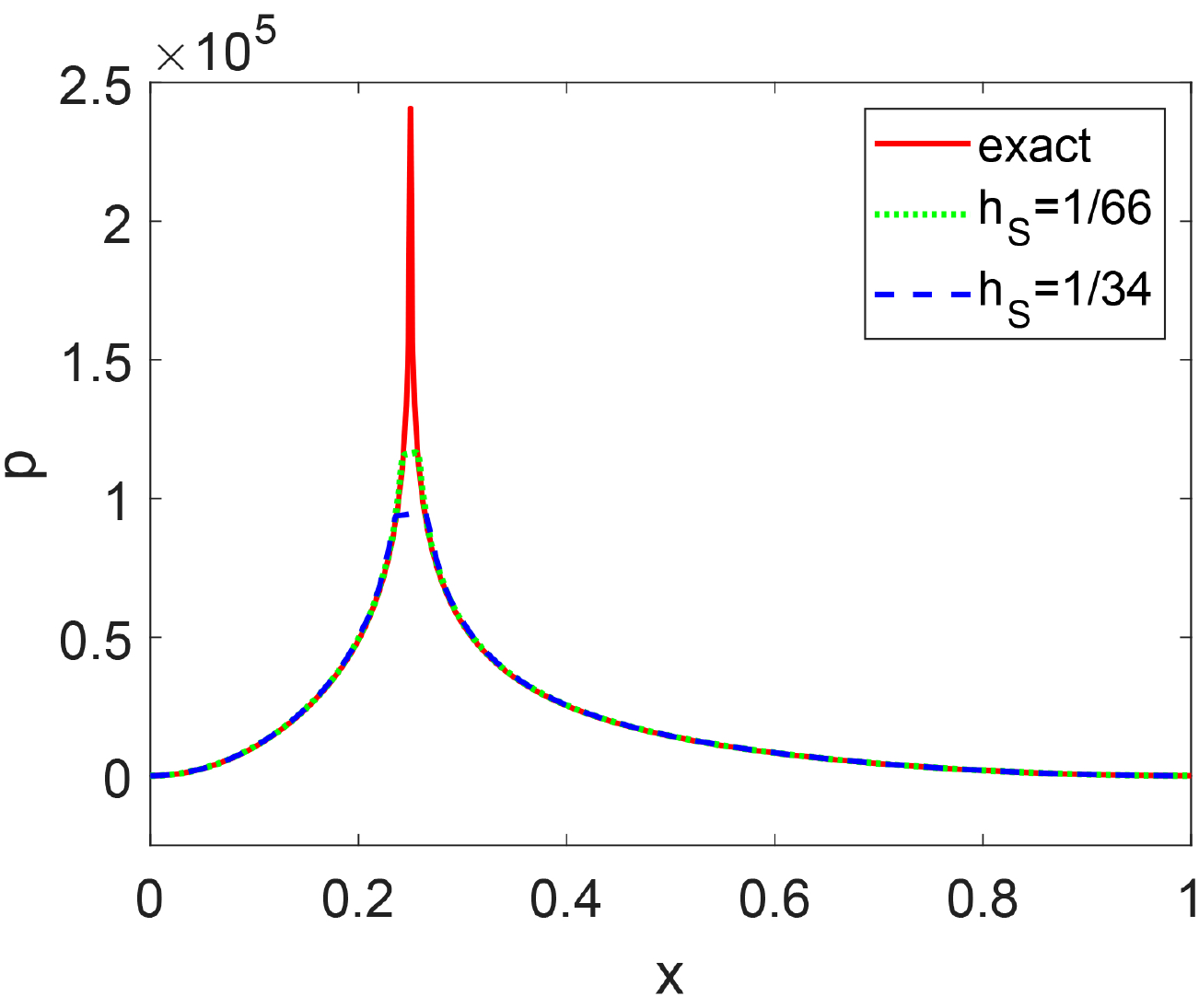}
		{ \small (a) \  Pressure at $t = 0.5 \, \pi/\beta$.}
	    \end{center}
		\end{minipage}
		\hspace{0.5cm}
		\begin{minipage}{6.1cm}
         \begin{center}
		   \includegraphics[scale=0.46]{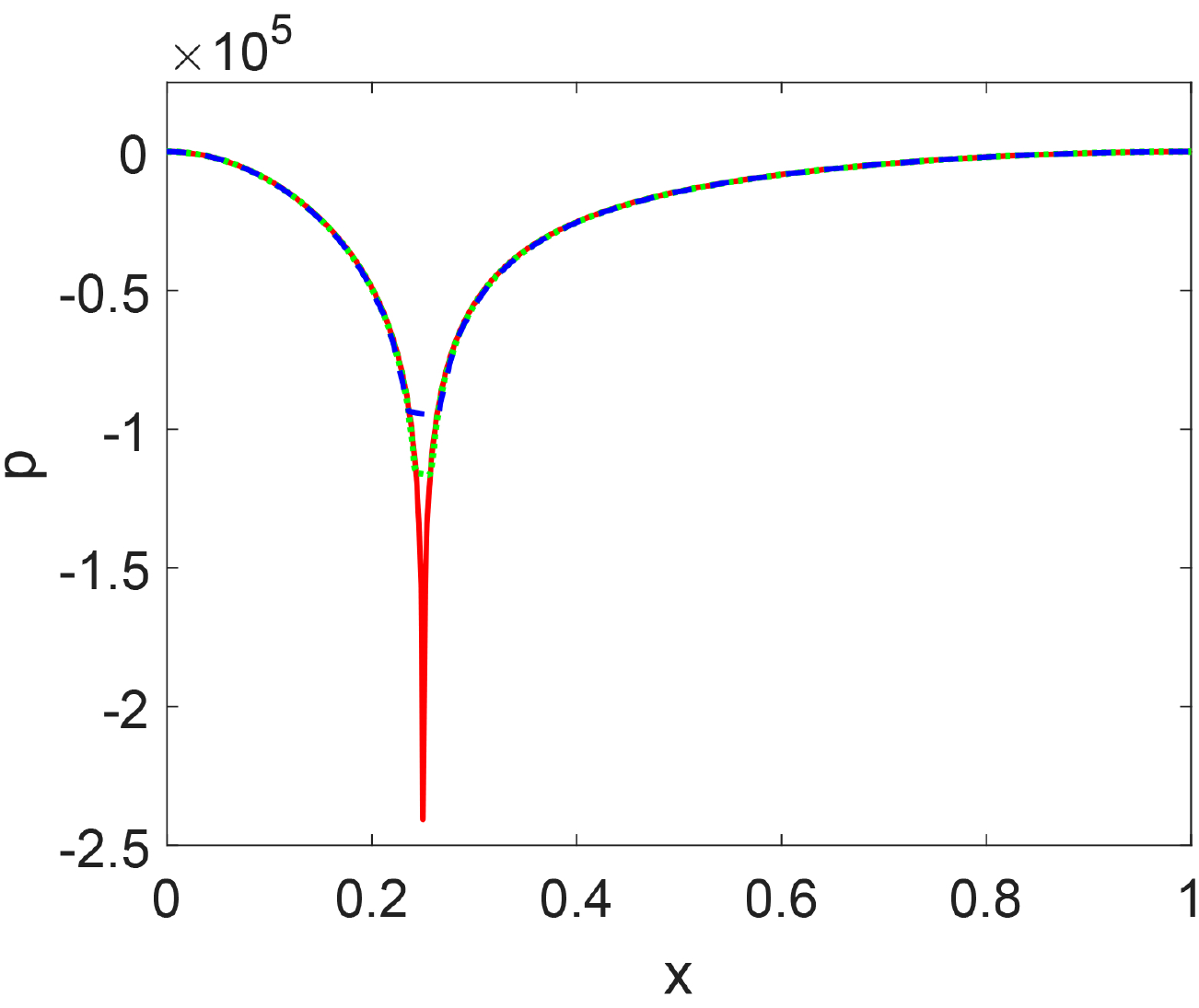}
	       { \small (b)  \ Pressure at $t = 1.5 \, \pi/\beta$.}
         \end{center}
		\end{minipage}		
		\vspace{0.2cm}
		\begin{minipage}{6.1cm}
			\begin{center}
				\includegraphics[scale=0.46]{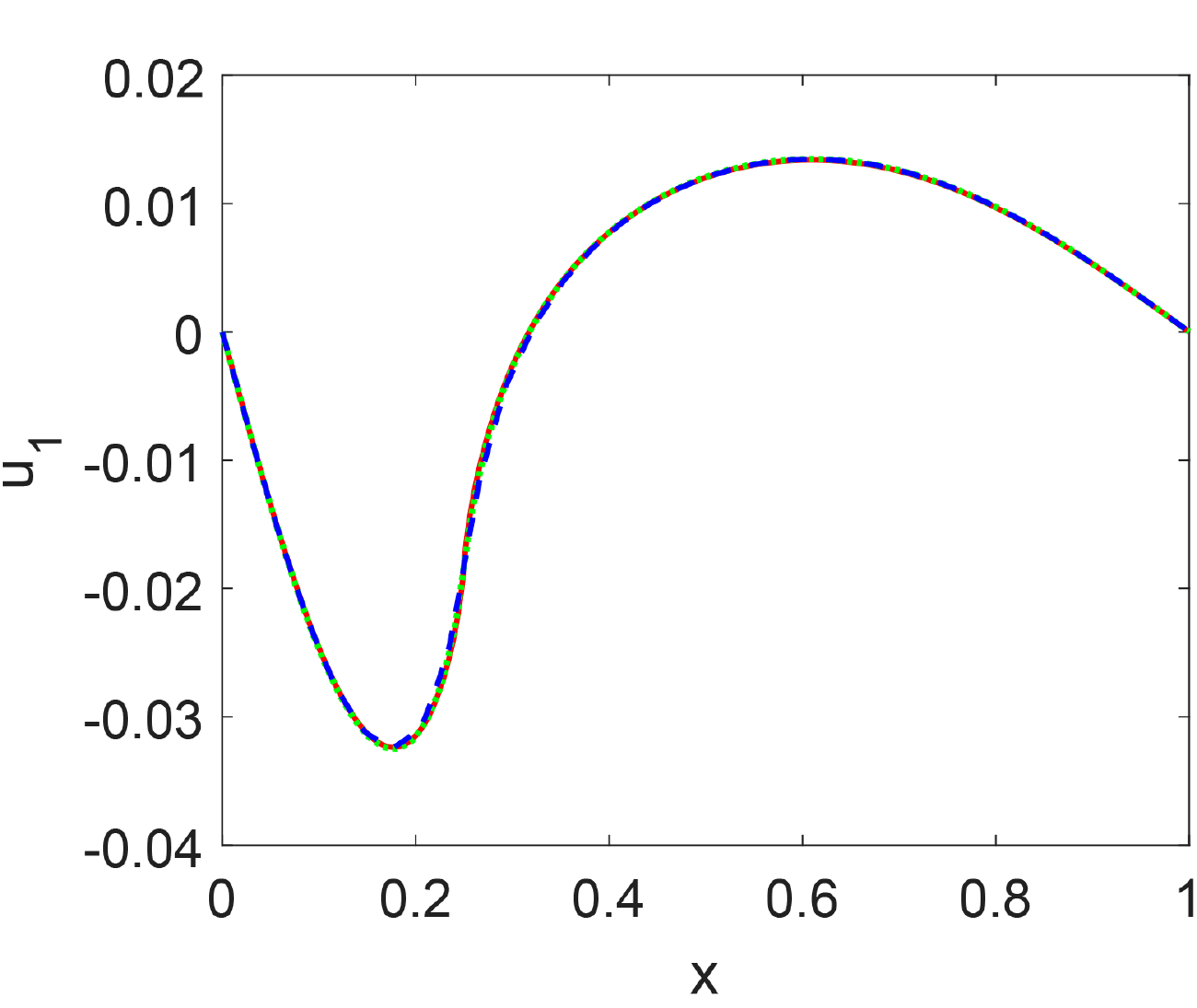}
				{ \small (c) \ $x$-displacement at $t = 0.5 \, \pi/\beta$.}
			\end{center}
		\end{minipage}
		\hspace{0.5cm}
		\begin{minipage}{6.1cm}
			\begin{center}
			\includegraphics[scale=0.46]{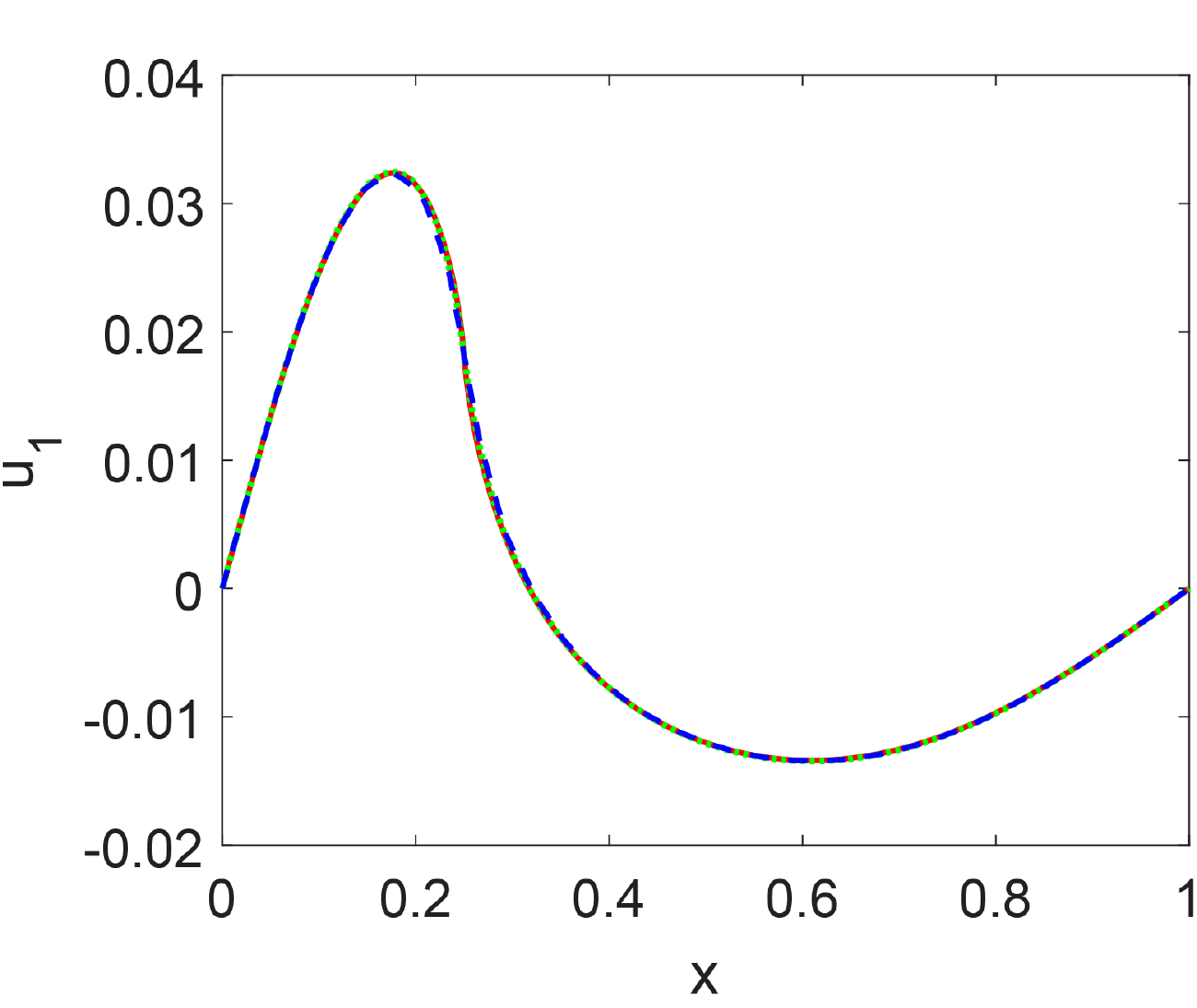}
			{ \small (d) \  $x$-displacement at $t = 1.5 \, \pi/\beta$.}
			\end{center}
		\end{minipage}				
	\end{center}
	\caption{ \small Comparison between exact solution and the numerical solutions along the diagonal line $\gamma$ from $(0,0)$ to $(1,1)$.}
	\label{fig:Barry_Mercer_4}
\end{figure}

\subsection{A 3D geometry with curved boundary}
\label{subsec:3D_example}
Next we demonstrate that our space-time method also works for 3D domains with curved boundary, which
underlines the advantages of an IGA approach. One notes the fact that for the 3D case the space-time cylinder is a four-dimensional object.
The crescent-shaped geometry with spatial mesh in Fig. \ref{fig : curved1} (a) is inspired by the porous structure of a human meniscus. From a biomedical viewpoint, the poor vascularization of the meniscus is one reason for premature osteoarthritis
in knee joints. On the other hand, the meniscus
tissue is highly hydrated (70-75\% water), and the frequent pressure changes during walking and running
are essential for the flow of nutrients and for fostering the regeneration capabilities.
We prescribe the following
parameter values to approximate the behaviour of such a fibro-cartilaginous material: $c_0=2.7 \cdot 10^{-10}, \ \lambda=472689,  \mu=183824, \p{\mathcal{K}}=1.5 \cdot 10^{-12} \cdot \p{I}, b=1$. The
method parameters are $r_p=r_u=2, \ r_T=1$ and $h_T = 1/18, \ T=0.5$. As boundary conditions we set the pressure to zero on the whole boundary except the flat bottom part of the meniscus, which can move in  horizontal directions but is fixed with respect to vertical movements. Both ends of the C-shaped domain are fixed, too, and a loading $$  \p{{\sigma}} \cdot \p{n}_x = f(t,z) \, (\frac{x}{\sqrt{x^2+y^2}},\frac{y}{\sqrt{x^2+y^2}},-1)^t,$$ with
$f(t,z)= 30000 \, \sin(\pi t) \, \sin \big( (1/0.0072) \, \pi\, z)$   is applied onto the upper surface.

\begin{figure}[h]
	\begin{minipage}{6.2cm}
		\includegraphics[scale=0.45]{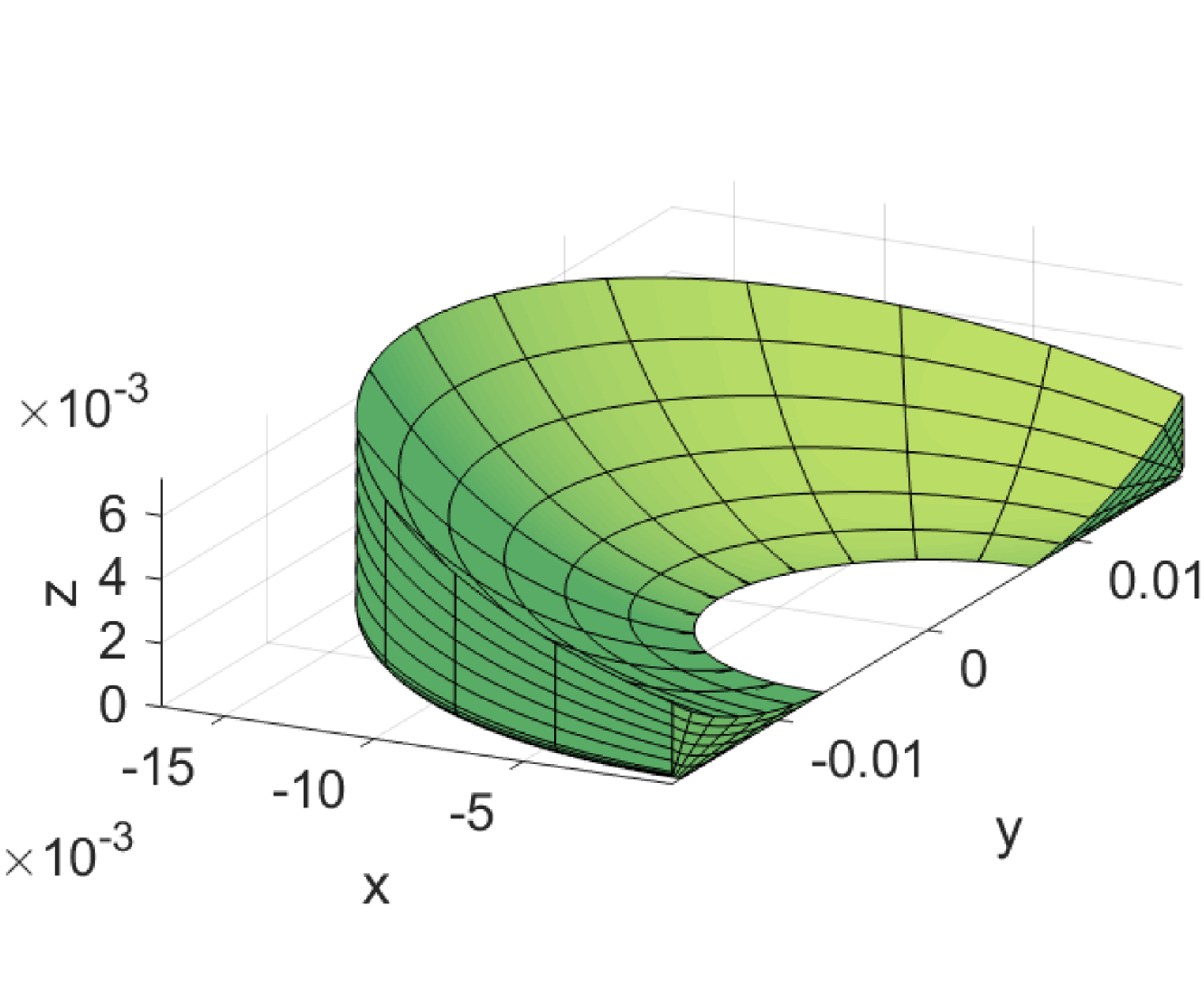}	
		{ \small \ (a) Spatial mesh of the 3D meniscus model.}
	\end{minipage}
	\hspace{1cm}
	\begin{minipage}{7.6cm} 		  
		\vspace{1.5cm}                  
		\includegraphics[scale=0.5]{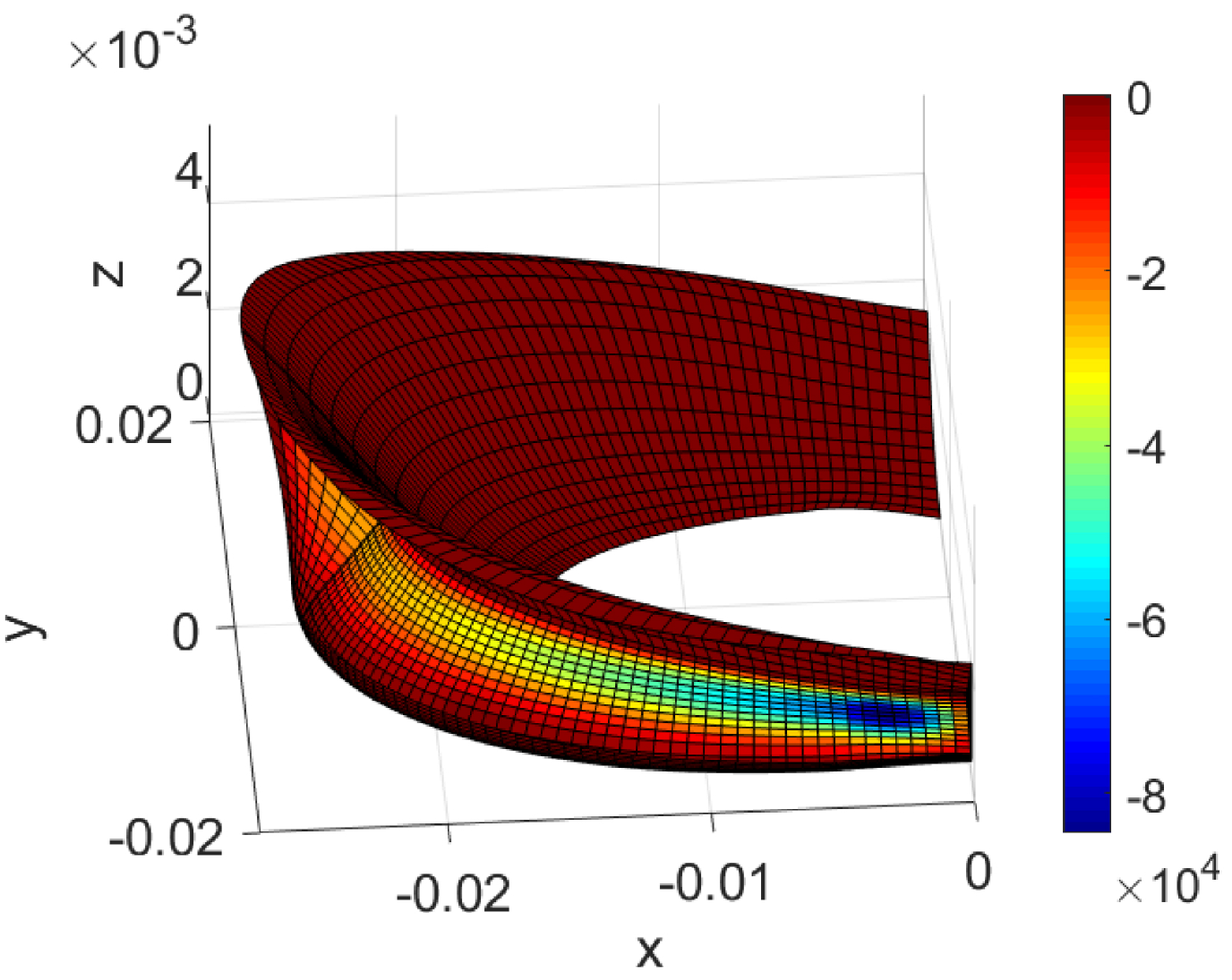}
		{ \small \ (b) Numerical pressure inside the deformed domain at the time $t = 0.5$ .}  \vspace{0.6cm} 
	\end{minipage}	
	\caption{ \small Mesh and numerical pressure for the case of a meniscus-type domain with curved boundaries.}
	\label{fig : curved1}
\end{figure}

Fig.~\ref{fig : curved1} displays the isogeometric mesh and a snapshot of the pressure distribution inside
the  fibro-cartilaginous material.

\subsection{Pressure oscillations and elastic locking}
\label{subsec:pressure_oscillations}
A major issue in solving Biot's equations is the occurrence of spurious pressure oscillations, mainly for low permeability, i.e., if $\norm{\p{\mathcal{K}}}<<1$. To simplify the discussion, we set
$\p{\mathcal{K}}=k \, \p{I}$ with constant $k$. Our numerical experiments show that especially small constrained specific storage coefficients along with low permeability may lead to a nonphysical behaviour.  To illustrate this, we plot in Fig.  \ref{fig :Terzaghi_osci}  approximate solutions to Terzaghi's problem  for small parameters $c_0=10^{-7}, \ k=10^{-7}$. As a result one 
observes oscillations despite the relatively fine spatial mesh ($h_S = 0.025$) for equal polynomial degrees $r_p=r_u=1$. These pressure oscillations are  well-known and can be handled by additional stabilization or discontinuous finite element methods.
There is also a connection to the locking effect in elasticity 
\cite{Braess}(Ch.~6). 
In the context of poroelasticity the pressure variable is more critical, but volumetric locking, i.e. the blocking of the displacements  in regions of low-compressible media, can be detected for  the Biot system, too.

The pressure variable can be stabilized by means of a mixed ansatz \cite{Biotmixed,on_the_cause_of_osci}. 
For standard IGA combined with implicit Euler in time,  one can show (using Theorem 5.2. in \cite{IGA3}) that  Taylor-Hood mixed spaces  satisfy  a Babu\v{s}ka-Brezzi inf-sup condition.  
We follow this idea of mixed spaces and increase the polynomial degree in the displacement variable by one. In
Fig.  \ref{fig :Terzaghi_osci} we show the result of above Terzaghi test case for $r_u=r_p+1$ in comparison with equal polynomial degrees. The Taylor-Hood ansatz results in an overshooting numerical solution, but  approximates the exact solution very precisely away from the  problematic boundary point $x=0$.
\begin{figure}[h]
	\begin{center}
		\begin{minipage}{7cm}
			\includegraphics[scale=0.5]{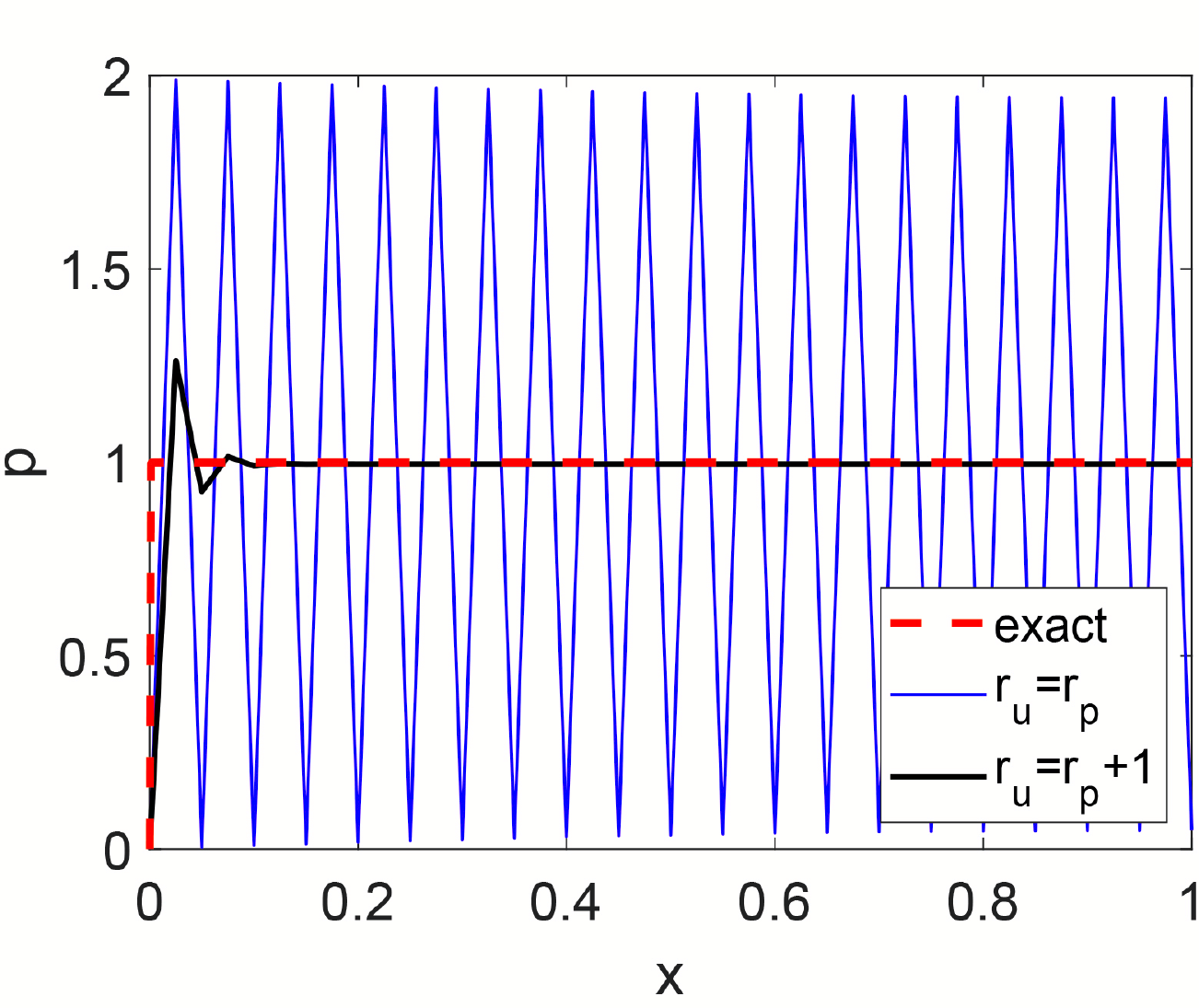}
			\caption{ \small   The Terzaghi problem with small permeability and storativity leads in the equal degree case to massive oscillations.}
			\label{fig :Terzaghi_osci}
		\end{minipage}
		\hspace{0.8cm}
		\begin{minipage}{7cm} 			                    
			\begin{tikzpicture}[scale=1.75]
			\draw[->,dashed] (0,0) to (2.33,0);
			\draw[->,dashed] (0,0) to (0,2.33);
			\draw[fill=gray,opacity=0.1] (0,0) -- (2,0) -- (2,2) -- (0,2) -- (0,0);
			\draw[fill=blue,opacity=0.1] (0,0.5) -- (2,0.5) -- (2,1.5) -- (0,1.5) -- (0,0.5);
			\node at (0.7,0.25) { \small $\lambda=1, \ k=1$};
			\node at (0.7,1.12) { \small $\lambda=1, $};
			\node at (0.7,0.88) { \small $k=10^{-8}$};
			\node at (0.7,1.75) { \small $\lambda=1, \ k=1$};
			\draw[thick,dashed] (1.5,0) to (1.5,2);
			\draw[blue, very thick] (2,0) to (2,2);
			\node[rotate = 90] at (1.65,1) {\small plot line};
			\draw[blue,very thick] (0,0) to (2,0);
			\draw[blue,very thick] (0,0) to (0,2);
			\draw[brown,very thick] (0,2) to (1,2);
			\draw[red,very thick] (1,2) to (2,2);
			\draw (1,2) to (1,2.1);
			\node[left] at (-0.1,2) {\small $1$};		
			\node[left] at (-0.1,0) { \small $0$};
			\node[left] at (0.03,2.14) { \small $y$};
			\node[below] at (2,-0.1) { \small $1$}; 
			\node[below] at (0,-0.1) { \small $0$};  
			\node[below] at (2.2,-0.05) { \small $x$};  
			\node[right] at (1.53,0.25) { \small $\gamma$};   
			\draw (-0.1,0) to (0,0);
			\draw (0,0) to (0,-0.1);
			\draw (-0.1,2) to (0,2);
			\draw (2,0) to (2,-0.1);
			\node[blue,left] at (2.92,0.9) {\small $ =\p{u} \cdot \p{n}_x$};
			\node[blue,left] at (2.52,0.7) {\small $ =0$};
			\draw[->,red,very thick] (1.35,2.6) to (1.35,2.1);
			\node[red,right] at (1.4,2.5) {\small $  \p{\sigma} \cdot \p{n}_x = - \p{n}_x  $};
			\node[blue,left] at (2.92,1.2) {\small $  \nabla_xp \cdot \p{n}_x $};
			\node[red,left] at (2.1,2.15) {\small $  p=0 $};
			\node at (-0.5,0.5) {$     $};
			\node[brown,left] at (0.93,2.5) {\small $   \p{\sigma} \cdot \p{n}_x =\p{0}, $};
			\node[brown,left] at (0.93,2.2) {\small $  p=0 $};
			\end{tikzpicture}
			\caption{ \small Boundary conditions and parameters for the low-permeable layer. For the second test case, i.e. low-compressible layer, the layer parameters are changed to $\lambda=10^{8}$ and $k=1$. }
			\label{fig: stability_test1}
		\end{minipage}
	\end{center}
\end{figure}  
Other numerical tests show further that using equal but higher polynomial degrees for both pressure and displacement 
is not really leading to a substantial improvement.  Hence a significant reduction of the oscillations without the need of very small mesh sizes can only be achieved with a mixed ansatz. 

We further illustrate  the effect of a mixed ansatz  also for the displacement variable by solving two test problems from \cite{on_the_cause_of_osci}. 
In both cases the  spatial domain is $\Omega=(0,1)^2$. 
First we place inside a  material with moderate parameters a low-permeable layer, and for the second case  we place there a low-compressible layer. More precisely, in the first case we have a region in which $k<<1$ is very small and in the other case we have analogously a large Lam\'{e} coefficient $\lambda >> \mu=1$. We set $c_0=0 , \mu=1  , \ b=1$; see also Fig. \ref{fig: stability_test1}. For the case with a low-compressible layer we change the parameter in the layer from $\lambda=1, \ k=10^{-8}$ to $\lambda=10^{8}, \ k=1.$

The first test case checks the reduction of pressure oscillations and the second one analyzes the elastic locking effect.  On the top edge of the domain  we apply a constant non-uniform normal load, namely $$\p{\sigma} \cdot \p{n}_x=\p{0}   \ \ \textup{for} \ x < 0.5  \ \   \ \textup{and} \  \ \ \p{\sigma} \cdot \p{n}_x=-\p{n}_x \ \  \textup{for} \ x \geq 0.5. $$  The computed solutions  at $t =1$ for zero  initial conditions and the two scenarios $1=r_p=r_u$  and $2=r_p+1=r_u$ are summarized in Fig.~\ref{fig :stability2}. Here we used a relatively fine uniform mesh with spatial mesh size $h_S=1/40$ as well as $h_T=0.2, \ r_T=1$ and plotted the solution along the line $(0.75,0)$  -  $(0,75,1)$. 
As already observed with Terzaghi's problem, we get a better result for the pressure solution and low-permeable layer if we use mixed polynomial degrees. The displacement is similar for the mentioned layer.  A look at the low-compressible layer case shows us  the smoothness of the pressure solution and the absence of oscillations. But the displacement variables differ for both cases. For degrees $r_p=r_u=1$ the displacement in the layer region is nearly constant. Consequently  the displacement is locked and we have the presence of elastic locking. But choosing $r_u=r_p+1=2$   the vertical displacement is more plausible. Thus the locking phenomenon is also damped  for mixed degrees. 
\begin{figure}[h]
	\begin{center}
		\begin{minipage}{7cm}
		\begin{center}
			\includegraphics[scale=0.5]{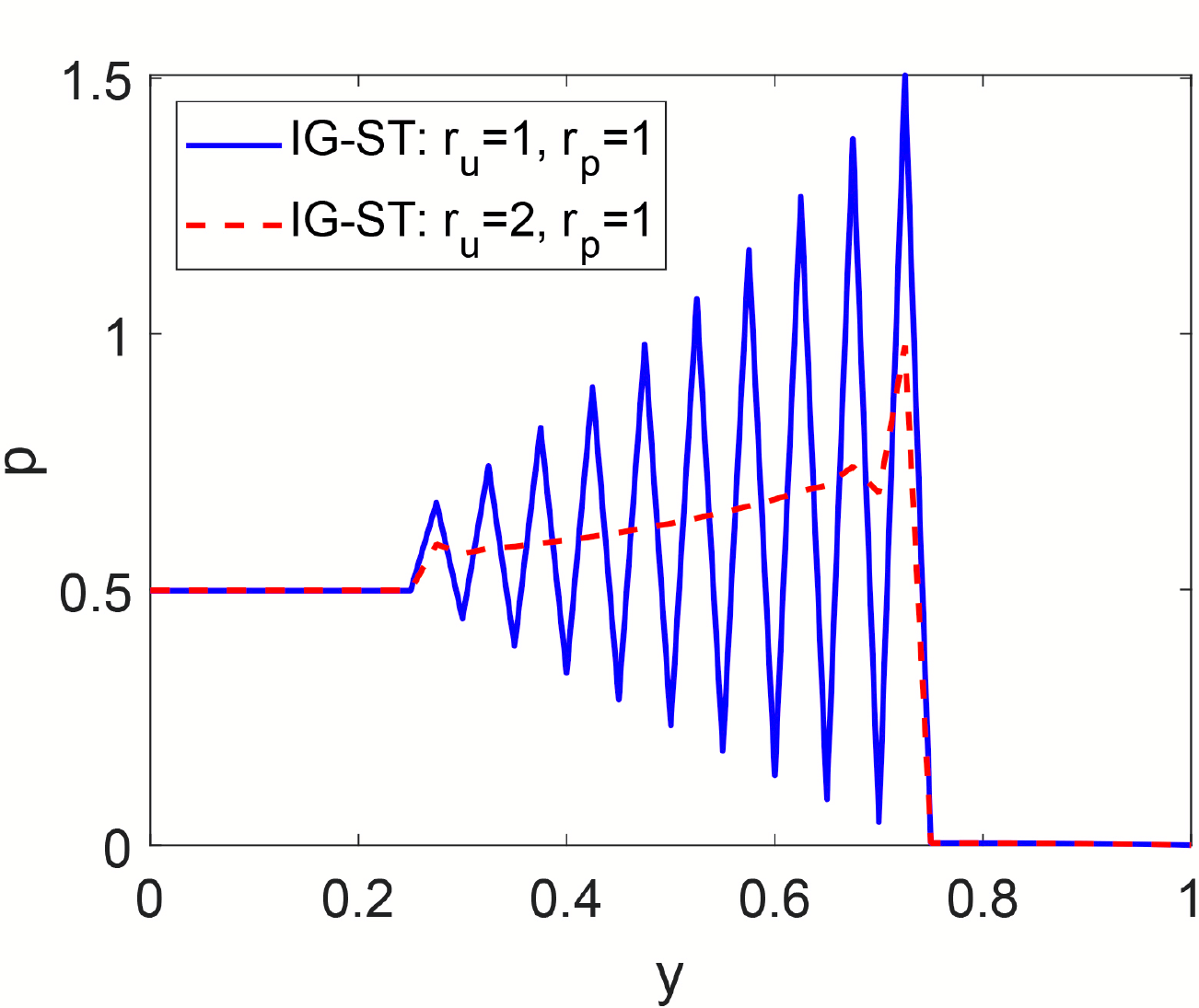}
			{(a) \ \small  Pressure  for   low-permeable layer.}
		\end{center}
		\end{minipage}
		\hspace{0.5cm}
		\begin{minipage}{7cm}
		\begin{center}
			\includegraphics[scale=0.5]{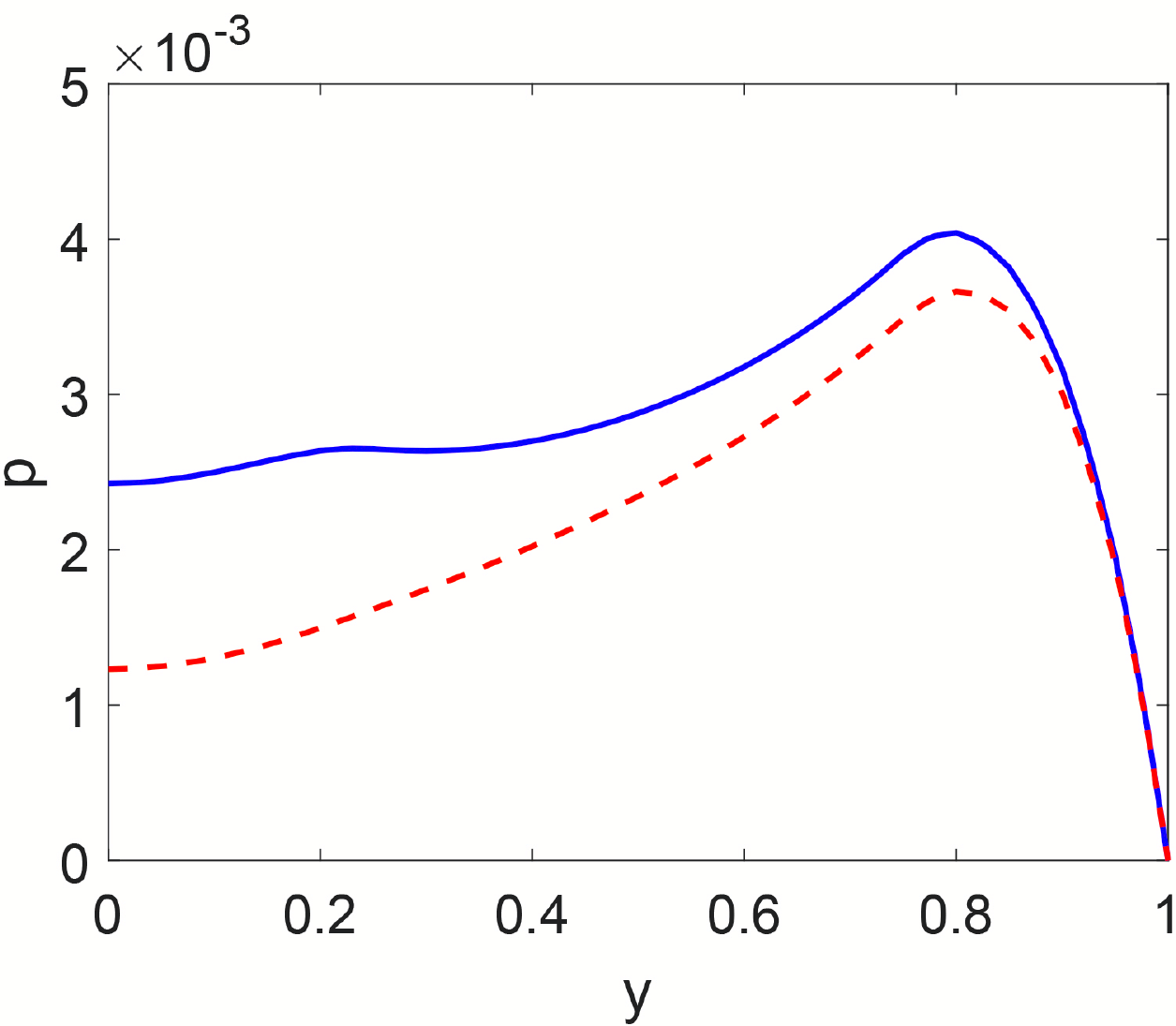}
			{(b) \ \small  Pressure  for   low-compressible layer.}
		\end{center}
		\end{minipage}
		\vspace{1cm}
		\begin{minipage}{7cm}
			\begin{center}
				\includegraphics[scale=0.5]{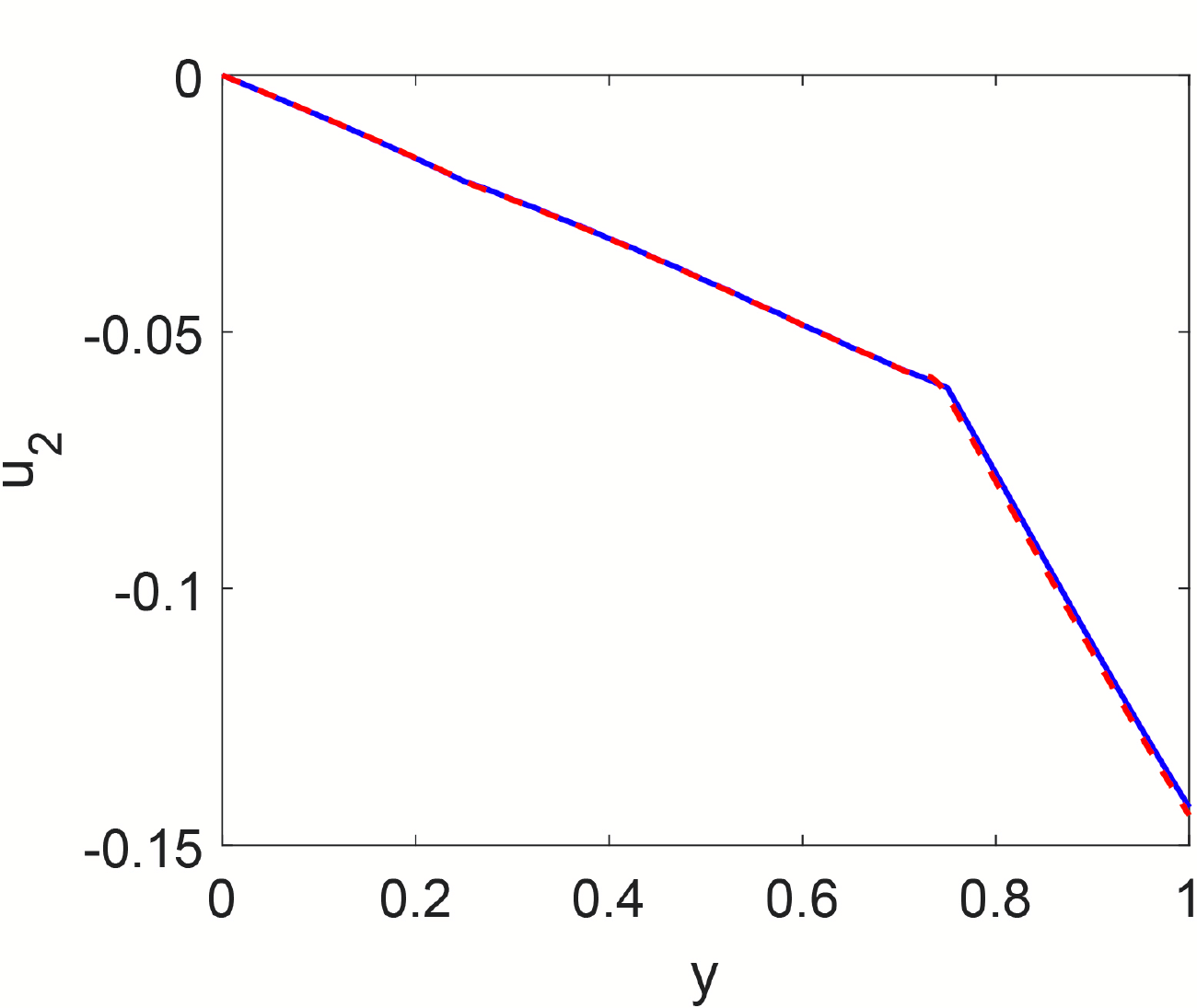}
				{(c) \ \small $y$-displacement for  low-permeable layer.}
			\end{center}
		\end{minipage}
		\hspace{0.5cm}
		\begin{minipage}{7cm}
			\begin{center}
				\includegraphics[scale=0.5]{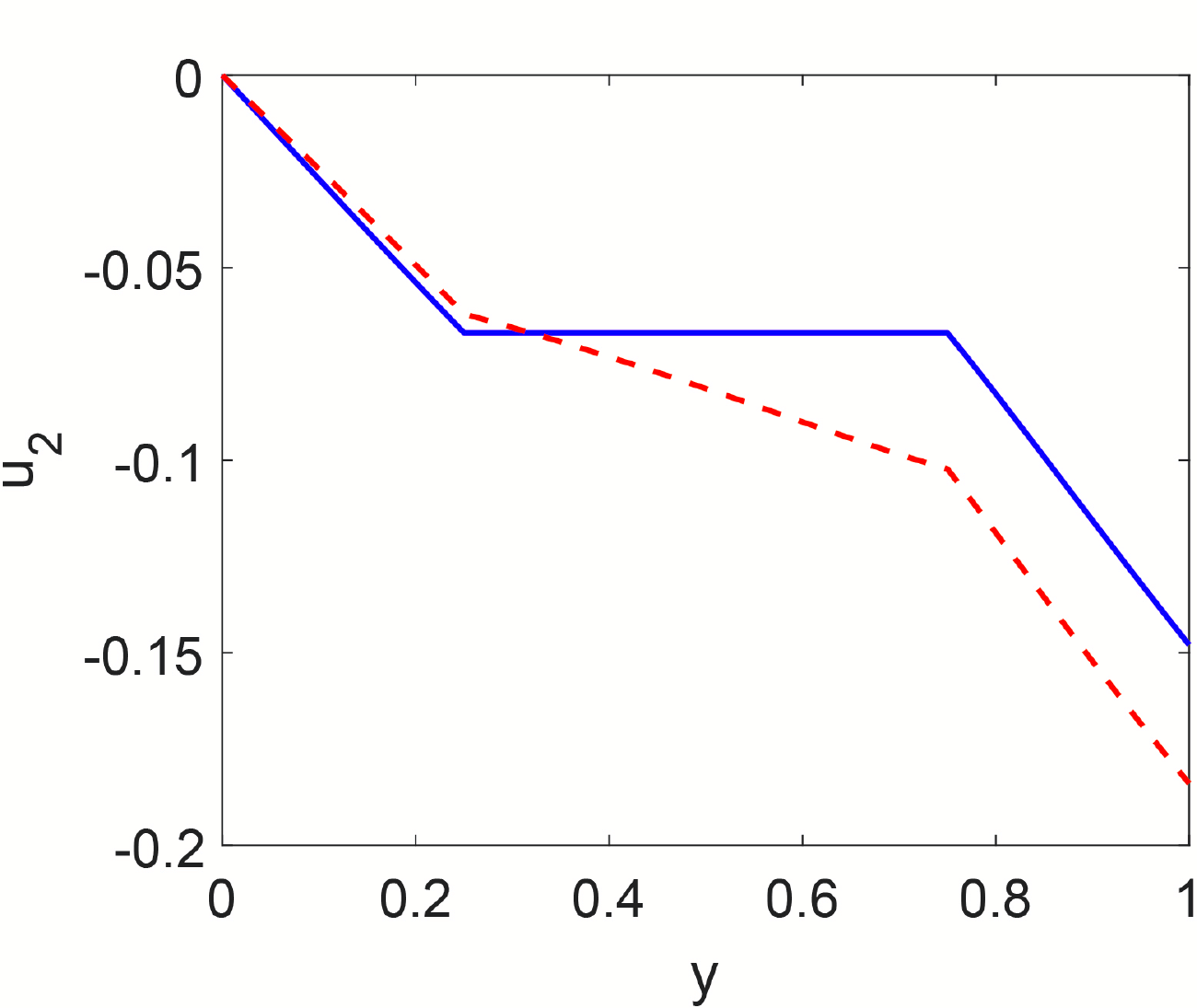}
				{(d)  \ \small  $y$-displacement for low-compr. layer.}
			\end{center}
		\end{minipage}
	\end{center}
	\caption{ \small  The results for the IG-ST method and the test cases of low-permeable  and low-compressible  layer along the plot line $(0.75,0)$ - $(0.75,1)$ at final time $T=1$.  Numerical solutions for mixed degrees are plotted in red, for equal degrees  in blue respectively. Mixed degrees reduce the pressure oscillations and the locking effect.    }
	\label{fig :stability2}
\end{figure}

The conclusion of this subsection is as follows. 
For the space-time method pressure oscillations but also  locking may  be present, mainly in the case of small permeability and large Lam\'{e} parameters. Mixed polynomial degrees stabilize the numerical solution, but nevertheless  discontinuous data  lead,  independent of the polynomial degrees, to local overshoots.


\section{Concluding remarks}
\label{sec:conclusion}
We have introduced and analyzed a novel isogeometric method for the Biot two-field system. It is based on a space-time discretization and allows to use the spline machinery to achieve arbitrary high convergence order, given sufficient regularity of the exact solution. By several numerical examples we have validated the theory and demonstrated the applicability, even for a
3D geometry. Moreover, the well-known problem of pressure oscillations has been addressed.

In our view, there are two major issues which should be treated in future work. On the one hand, although mixed methods lead to substantial improvements, a closer look at the pressure instability from a theoretical point of view might be advisable in the 
context of the space-time approach.
On the other hand it is desirable to consider the linear algebra for the resulting large-scale linear system in detail. Eventually, this may make the space-time approach also very competitive with respect to computing times. 



\footnotesize
\bibliographystyle{siam}
\bibliography{References}

\end{document}